\documentclass[english]{article}
\usepackage[T1]{fontenc}
\usepackage[latin9]{inputenc}
\usepackage{babel}
\usepackage{amsmath}
\usepackage{esint}
\usepackage{amsfonts}
\usepackage{amssymb}
\usepackage{amsthm}
\usepackage{mathrsfs}
\usepackage{amscd}
\usepackage{caption}
\usepackage{subcaption}
\usepackage{graphicx} 
\usepackage{fullpage}

\newcommand{\R}{\mathbb{R}}

\newcommand{\Z}{\mathbb{Z}}
\newcommand{\ul}{\underline}

\newtheorem{theorem}{Theorem}
\newtheorem{prop}[theorem]{Proposition}
\newtheorem{claim}[theorem]{Claim}
\newtheorem{cor}[theorem]{Corollary}
\newtheorem{lemma}[theorem]{Lemma}

\theoremstyle{definition}
\newtheorem{definition}{Definition}
 
\theoremstyle{remark}
\newtheorem{remark}{Remark}

\begin{document}

\markboth{Nicolas Petit}
{Finite-type invariant of order one of long and framed virtual knots}

\title{Finite-type invariants of long and framed virtual knots}

\author{Nicolas Petit}

\maketitle
\begin{center}Oberlin College, 10 N Professor St\\
King 205, Oberlin, OH 44074, USA \\
npetit@oberlin.edu\end{center}

\begin{abstract}
We generalize three invariants, first discovered by A. Henrich, to the long and/or framed virtual knot case. These invariants are all finite-type invariants of order one, and include a universal one. The generalization will require us to extend the notion of a based matrix of a virtual string, first introduced by V. Turaev and later generalized by Henrich, to the long and framed cases.
\end{abstract}

\section{Introduction}
In her 2010 paper \cite{2}, A. Henrich discussed three invariants of virtual knots, that were finite-type invariants (FTIs for short) of order one, building on work by Fiedler \cite{8} and Chernov \cite{9}.
She proved that each of them was stronger than the previous one, and that the last invariant she defined (the so-called ``glueing invariant'') was in a certain sense the universal FTI of order $\leq 1$.
In order to do so she introduced singular virtual strings, expanding on the concepts of virtual strings and the associated based matrices, which were first studied by Turaev in \cite{5}.

This paper aims to extend the results of \cite{2} for framed virtual knots and long virtual knots, both framed and unframed.
To do so, we will need to give an appropriate definition of a framed virtual string and of a long virtual string, and check that the constructions in Henrich's paper still hold in our case.
We will see that all three of Henrich's invariants generalize naturally, and that the glueing invariant is the universal FTI of order $\leq 1$ (in the same sense as in Henrich's paper) for the virtual knot categories we're considering.

In the framed case, generalizing said results requires us to study the theory of based matrices of framed virtual strings, which generalizes the theories developed in \cite{5} and \cite{2}.
It turns out that working in the framed category creates an indeterminacy in the based matrices due to the possible application of a Whitney belt trick.
This indeterminacy, which is also present when looking at singular based matrices, was absent from the unframed case, 
The generalization to the long virtual knot case appears to be more natural, and old examples are easily repurposed for the long category. 
We can then study the behavior of the invariants under connected sum and closure; this leads to interesting results for the polynomial invariant.
Moreover, the polynomial invariant admits a slight generalization in the long case, called the ordered polynomial invariant.
Finally, to prove universality we need to generalize the notion of a based matrix to the case of long virtual strings; in order to do so, we will look at the interpretation of long virtual knots as knots in thickened surfaces with boundary \cite{7} and at Turaev's original idea for based matrices as the homological intersection matrix of the minimal surface of the long knot.

This paper is structured as follows: section \ref{section2} recalls the definitions of virtual knots, flat virtual knots and finite-type invariants; particular care is given to the interpretation of long knots as knots in thickened surfaces with distinguished boundary. 
Section \ref{section3} recalls the invariants from \cite{2} and generalizes them to the framed case. 
Section \ref{section4} recalls and extends the construction of based matrices; in particular, section \ref{section5} proves that the glueing invariant is strictly stronger than the smoothing invariant in all the categories we considered.

\section{Background}
\label{section2}
\subsection{Virtual knots, framed and flat}

Virtual knots were introduced by Kauffman in \cite{3}.
We can think about them in three different ways: we can interpret them as knots in thickened surfaces up to the addition/removal of handles; as classes of Gauss diagrams; or as knot projections with three types of crossings (classical positive, classical negative and virtual) up to Reidemeister moves (both classical and virtual).
This paper will focus on framed virtual knots, which are knots with a non-vanishing normal vector field; diagrammatically this category uses all the Reidemeister moves (classical and virtual) except for classical Reidemeister one, which gets replaced by the move shown in Fig. \ref{framedr1}. Note in particular that virtual Reidemeister one is still allowed. We will sometimes call the usual virtual knot theory ``unframed'' (in contrast with the ``framed'' version).

\begin{figure}[!h]
\centering
\includegraphics[scale=.25]{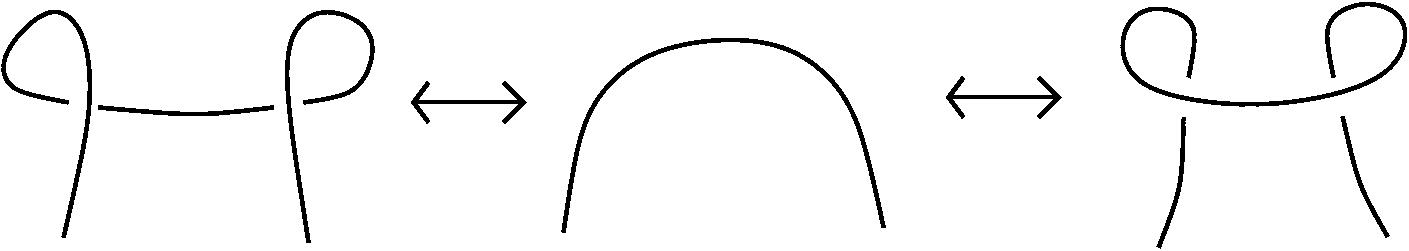}
\caption{The framed Reidemeister one move.}
\label{framedr1}
\end{figure}

Virtual knots can be represented as equivalence classes of Gauss diagrams, which are a way of encoding the information contained in the knot diagram.
Given a virtual knot diagram with $n$ crossings, we get the Gauss diagram associated to it by taking a counter-clockwise oriented circle and $2n$ distinct points on it, divided in pairs.
Each pair of points will correspond to a crossing, and will be given a signed arrow; interpreting the circle as a parametrization of the knot, each point corresponds to either going over or under the crossing point.
The signed arrow will then go from the preimage of the overcrossing to the preimage of the undercrossing, and carry the sign of the crossing.
We can reverse the process to get a knot diagram out of a Gauss diagram.
While the reverse process can yield different knot diagrams, all those diagrams belong to the same isotopy class, so they represent the same virtual knot.
Classical knots naturally sit inside virtual knots as the collection of Gauss diagrams that admit a knot diagram representation that doesn't have virtual crossings.
We will call a knot that is virtual but not classical \emph{properly virtual}.

If instead of looking at isotopy classes we look at homotopy classes of virtual knots we get the theory of \emph{flat virtual knots} (or \emph{virtual strings}, in Turaev's terminology).
Diagrammatically this is equivalent to forgetting whether a classical crossing is positive or negative, as we can deform one into the other; or to adding the crossing change (CC) move, that lets us change a positive crossing into a negative one and vice versa, pictured in Fig. \ref{CCmove}.

\begin{figure}[!h]
\centering
\includegraphics[scale=.15]{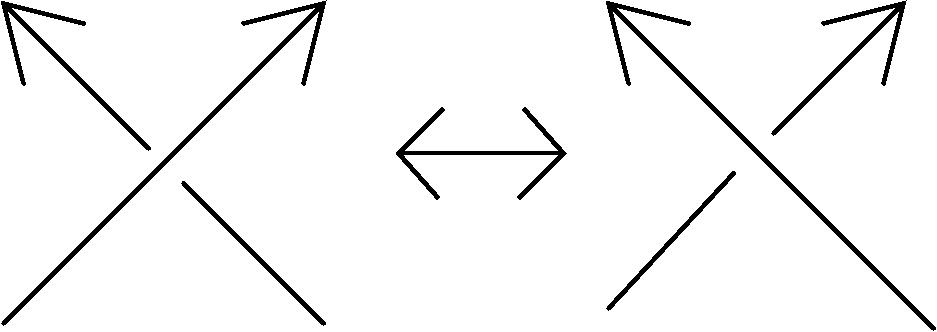}
\caption{The crossing change (CC) move.}
\label{CCmove}
\end{figure}

One big difference between classical and virtual knot theory is the number of homotopy classes: there is only one homotopy class of classical knots, corresponding to the unknot, as we can always transform any knot into the unknot by switching some crossings.
This is no longer true for virtual knots, and there are infinitely many homotopy classes of virtual knots.
Diagrammatically we represent flat virtual knots as knot diagrams with two types of crossings (classical and virtual) with an appropriate set of Reidemeister moves, which we get from the moves of virtual knot theory by forgetting the over/under crossing information.
Since the over/under information is not relevant anymore, our classical crossings will simply look like double points of the projection.
The flat Reidemeister moves then comprise all the Reidemeister moves with only virtual crossings, plus the ones in Fig. \ref{flatreidemeistermoves}.

\begin{figure}[!h]
\centering
\includegraphics[scale=.15]{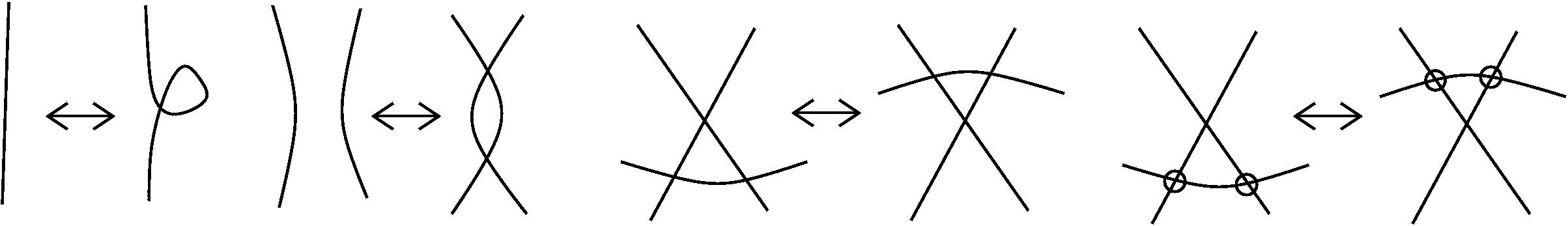}
\caption{Some of the flat Reidemeister moves.}
\label{flatreidemeistermoves}
\end{figure}

Clearly we can define virtual links and flat virtual links in a similar way, as (flat) virtual knots with multiple components, up to the appropriate Reidemeister moves.
A link can also be represented as a Gauss diagram; it will have as many core circles as the number of components of the link.
Arrows whose head and tail belong to the same component will represent self-crossings of the component, while arrows whose head and tail belong to different components will represent crossings involving two distinct components. 

To generalize the three invariants from \cite{2} to framed virtual knots we need to introduce the notion of a flat framed virtual knot.
The notion is very natural: they correspond to framed virtual knots modulo the CC move, and to homotopy classes of framed virtual knots.
Diagrammatically they correspond to flat virtual knots, where the first classical Reidemeister move has been replaced by the flat version of Fig. \ref{framedr1}, shown in Fig. \ref{flatframedR1}.

\begin{figure}[!h]
\centering
\includegraphics[scale=.25]{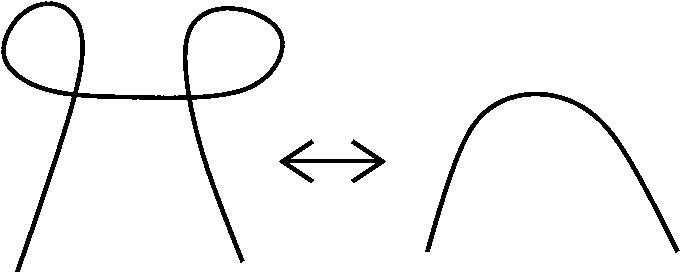}
\caption{The flat, framed Reidemeister one move.}
\label{flatframedR1}
\end{figure}

\subsection{Long virtual knots}
Since this paper will extend the constructions of \cite{2} to the long virtual knot case, let us review some notions and constructions regarding long virtual knots, with particular attention to the interpretation of long knots as knots in thickened surfaces.
More details on the subject can be found in \cite{7}.

\begin{definition}[Long virtual knot diagrams]
A \emph{long virtual knot diagram} is an immersion $K\colon \R\to \R^2$ which coincides with the $x$-axis outside of some closed ball centered at the origin and whose double points are decorated either as a classical crossing (with over/under information) or as a virtual crossing. 
A long virtual knot is an equivalence class of long virtual knot diagrams, where two knots in the equivalence class are related by a finite sequence of (classical and virtual) Reidemeister moves.
If the diagram has no virtual crossings we call it a \emph{long (classical) knot diagram}.
If a long virtual knot has a long classical diagram in its equivalence class we call it a \emph{long classical knot}.
\end{definition}

Let $\Sigma$ be a compact, connected surface with boundary and one or two distinguished boundary components, $g(\Sigma)$ its genus, $c(\Sigma)$ the number of boundary components (including the distinguished ones) and $num(\Sigma)$ the number of distinguished components.
Note that the Euler characteristic of $\Sigma$ is $\chi(\Sigma)=2-2g(\Sigma)-c(\Sigma)$.

\begin{definition}[Long knots in thickened surfaces]
Let $\Sigma$ be as above. If $C$ is a distinguished boundary component of $\Sigma$, let $C\times I \subset \Sigma \times I$ be an annulus.
A long knot in a thickened surface is a smooth embedding $\tau\colon I\to \Sigma\times I$ satisfying:
\begin{enumerate}
\item $im(\tau)\cap \partial(\Sigma\times I)=\{\tau(0), \tau(1)\}$;
\item $\tau(0), \tau(1)\in int(C\times I), \tau(0)\neq \tau(1)$ if $num(\Sigma)=1=\{C\}$;
\item $\tau(0)\in int(C_1\times I), \tau(1)\in int(C_2\times I)$ if $num(\Sigma)=2=\{C_1, C_2\}$.
\end{enumerate}
We denote the long knot by the pair $(\Sigma, \tau)$.
\end{definition}


\begin{definition}
Two longs knots on thickened surfaces $(\Sigma, \tau_1), (\Sigma, \tau_2)$ are said to be equivalent if we can obtain one from the other by a finite sequence of (classical) Reidemeister moves on $\Sigma$.
Two long knots $(\Sigma_1, \tau_1)$ and $(\Sigma_2, \tau_2)$ are said to be \emph{elementary equivalent} if there is a connected oriented surface $\Gamma$ and orientation-preserving embeddings $g_1\colon \Sigma_1\to \Gamma$, $g_2\colon \Sigma_2\to\Gamma$ such that $(\Gamma, g_1(\tau_1))$ and $(\Gamma, g_2(\tau_2))$ are equivalent.
We denote elementary equivalence by $(\Sigma_1,\tau_1)\sim_e(\Sigma_2,\tau_2)$.
Two long knots $(\Sigma_1, \tau_1)$ and $(\Sigma_2, \tau_2)$ are \emph{stably equivalent} if there is a finite sequence 

\begin{equation}(\Sigma_1,\tau_1)=(\Gamma_0, \gamma_0)\sim_e(\Gamma_1,\gamma_1)\sim_e\cdots\sim_e(\Gamma_n,\gamma_n)=(\Sigma_2,\tau_2)\end{equation}
of elementary equivalences.
We denote stable equivalence by $(\Sigma_1,\tau_1)\sim_s(\Sigma_2,\tau_2)$.
\end{definition}

We will now illustrate the correspondence between long virtual knots and stability classes of long knots on thickened surfaces.
To do so, we construct a ``band-pass'' presentation of a long virtual knot diagram as follows.
First, draw the long knot diagram $K$ in $\R^2$, and compactify $\R^2$ into $S^2=\R^2\cup \{\infty\}$.
Let $U$ be a neighborhood of $\{\infty\}$ in $S^2$, small enough so that it only intersects the long knot where the embedding is a line, and consider the long knot as an embedding into $S^2\setminus U$.

\begin{figure}[!h]
\centering
\includegraphics[scale=.12]{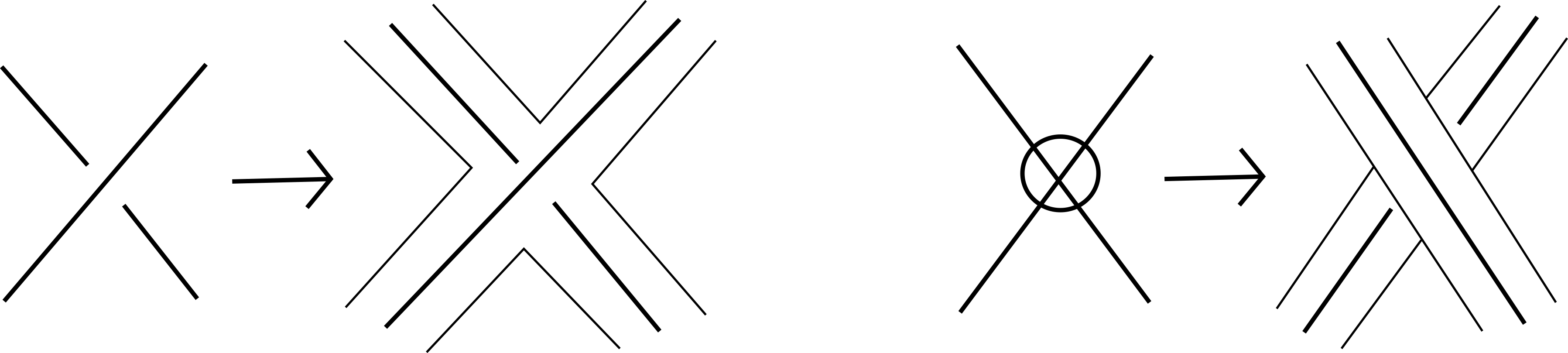}
\caption{How to create the band presentation of a virtual knot.}
\label{bandpresentation}
\end{figure}

Now replace a neighborhood of every classical crossing by a cross of bands, a neighborhood of every virtual crossing by a pair of non-intersecting bands, replace every arc of the knot by a band, the circle $\partial U$ by a circular band, and glue it all together. Finally, draw the knot on the bands, see Fig. \ref{bandpresentation}.
We get a connected, oriented surface $\Sigma_K$ with one distinguished boundary component $\partial U$, and denoting by $\tau_K$ the embedding of our long knot into $\Sigma_K$ we have that $(\Sigma_K,\tau_K)$ is a long knot in a thickened surface that represents the long knot diagram $K$.

The reason why we need to consider stability classes of long knots in thickened surfaces is to make sure that the inverse map is well defined.
More about the inverse map can be found in \cite{7}.
The construction of the band-pass presentation of a long virtual knot will be relevant in section \ref{basedmatriceslongstrings}, when talking about based matrices of flat long virtual knots.

\begin{remark}
For classical knots, there is a 1-1 correspondence between classical knots and long classical knots.
Fixing a point in our closed knot diagram that is not a crossing, we obtain a long knot by considering this fixed point as the point at infinity of $S^2=\R^2\cup\{\infty\}$.
This process is well defined as, no matter where along our knot we pick the fixed point, the long knots we obtain from the process are all isotopic. 

This is unfortunately not true for virtual knots.
The presence of the forbidden moves (see \cite{3}) implies that cutting a closed virtual knot along two different points might yield non-isotopic long virtual knots, see Fig \ref{longvirtnonisotopic}, whose two knots are distinguished by an invariant of \cite{1}.
\end{remark}

\begin{figure}[!h]
\centering
\includegraphics[scale=.15]{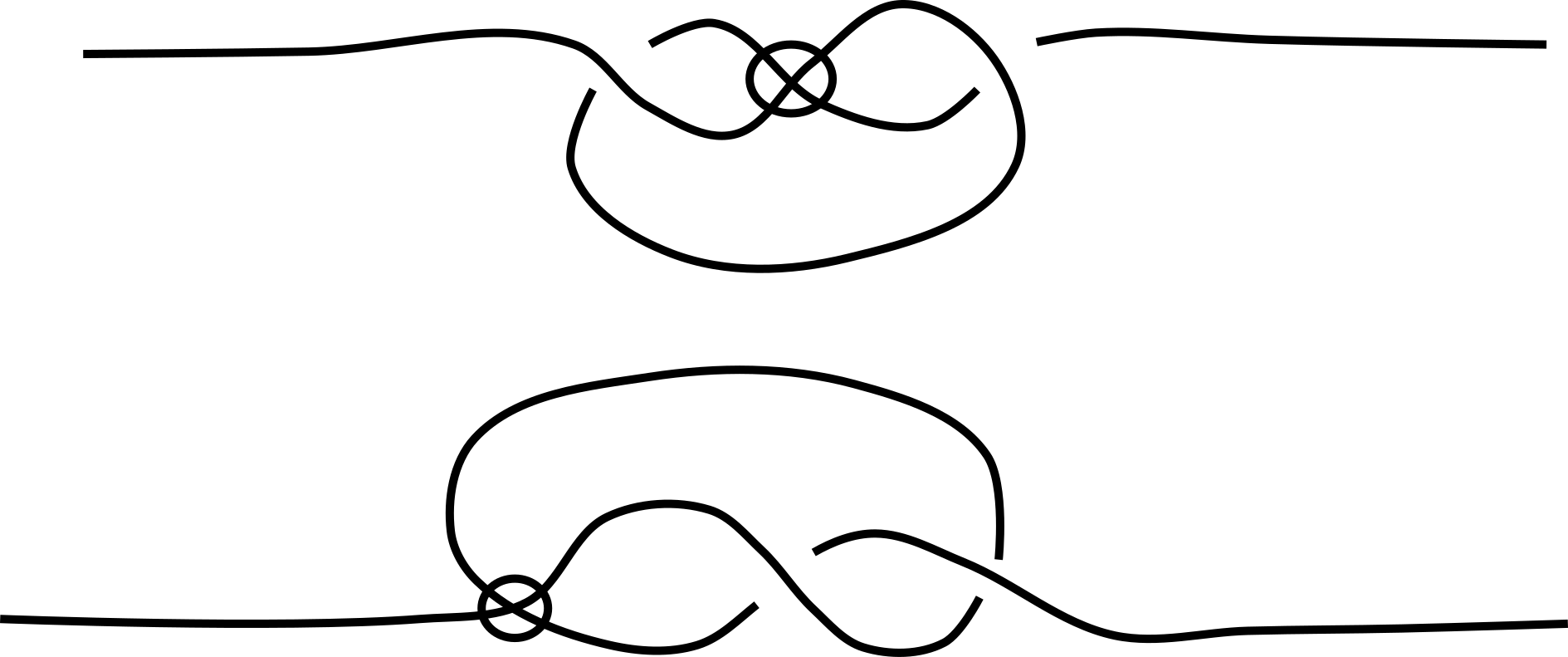}
\caption{Two non-isotopic long virtual knots with isotopic closure.}
\label{longvirtnonisotopic}
\end{figure}

\begin{remark}
\begin{enumerate}\item At the Gauss diagram level, long virtual knots are simply Gauss diagrams with a distinguished point on $S^1$ that represents the point at infinity.
Moreover, in any allowed deformation of the Gauss diagram, no arrowhead or arrow tail can cross the distinguished point.
The rest of the construction is unchanged from the closed case.
\item If we replace Reidemeister move one by its framed version (see Fig. \ref{framedr1}), we get the theory of framed long virtual knots, which are long virtual knots with a non-vanishing normal vector field.
\end{enumerate}
\end{remark}

\begin{remark}
\label{virtualstringdef}
If we consider long virtual knot diagrams modulo Reidemeister moves and the CC move (see Fig. \ref{CCmove}), we get the theory of flat long virtual knots.
Since we will be working a lot with this theory, let us see how the CC move affects the Gauss diagram, both for closed and long virtual knots.
Because the over/under information is now irrelevant, we need to reimagine how to encode the flat virtual knot with a structure that is similar to the Gauss diagram.
We follow \cite{5}, and simply treat every crossing as if it were positive and drop the signs from the arrows.

What this means in practice is that, to get a diagram for a flat virtual knot starting from the Gauss diagram of its virtual knot, we flip all the negative arrows and remove all the signs, see Fig. \ref{virtualstringfromgauss}.
If the Gauss diagram represented a long virtual knot, we keep the distinguished point in its position and give it the same ``point at infinity'' meaning.
We call this diagrammatic representation of a flat knot the \emph{virtual string} associated with the knot.
By abuse of language we will sometimes refer to the flat virtual knot as its virtual string.
\end{remark}

\begin{figure}[!h]
\centering
\includegraphics[scale=.25]{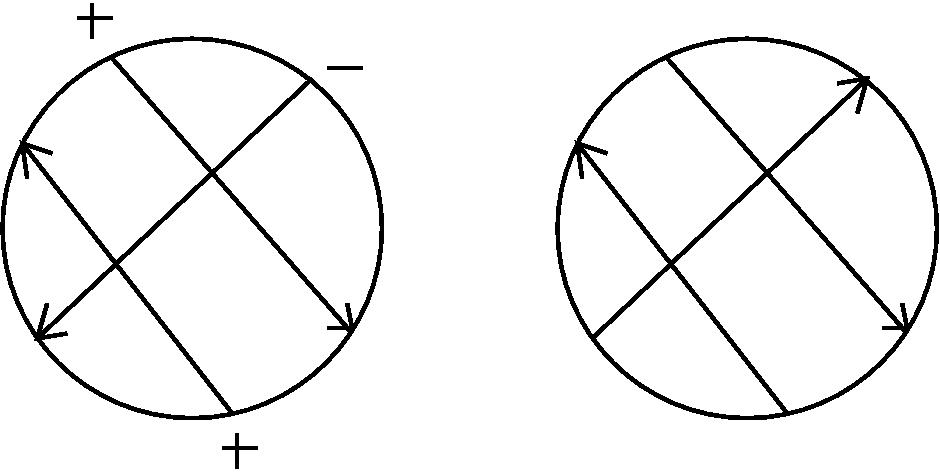}
\caption{The Gauss diagram of the virtual figure eight knot (left) and the virtual string of the flat virtual figure eight knot.}
\label{virtualstringfromgauss}
\end{figure}

\begin{figure}[!h]
\centering
\includegraphics[scale=.175]{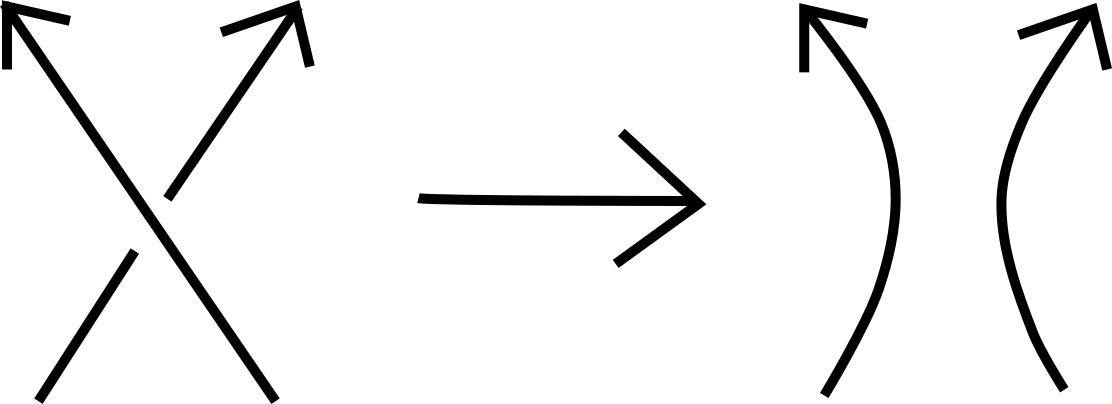}
\caption{Smoothing a crossing according to orientation.}
\label{crossingsmoothing}
\end{figure}

\begin{remark}
We will often in this paper talk about smoothing a crossing in a knot. 
All of our smoothings are done according to orientation, see Fig. \ref{crossingsmoothing}.
Note that smoothing a crossing in a long virtual knot splits the knot into two components, a closed one and an open one.
We will call the result of such a smoothing a \emph{two-component virtual tangle}.
While a more general theory of virtual tangles exists, it is beyond the scope of this paper.
Similarly, smoothing a two-component virtual link/tangle according to orientation at a crossing involving both components yields a virtual knot.\end{remark}

\subsection{Finite-Type Invariants}
\label{finitetypeinvariants}

In the classical case, a \emph{finite-type invariant} (or \emph{Vassiliev invariant}) is the extension of a knot invariant $\nu$ to the class of knots with transverse double points, given by the expression in Fig. \ref{doublepointresolution}.

\begin{figure}[h]
\centering
\includegraphics[scale=.07]{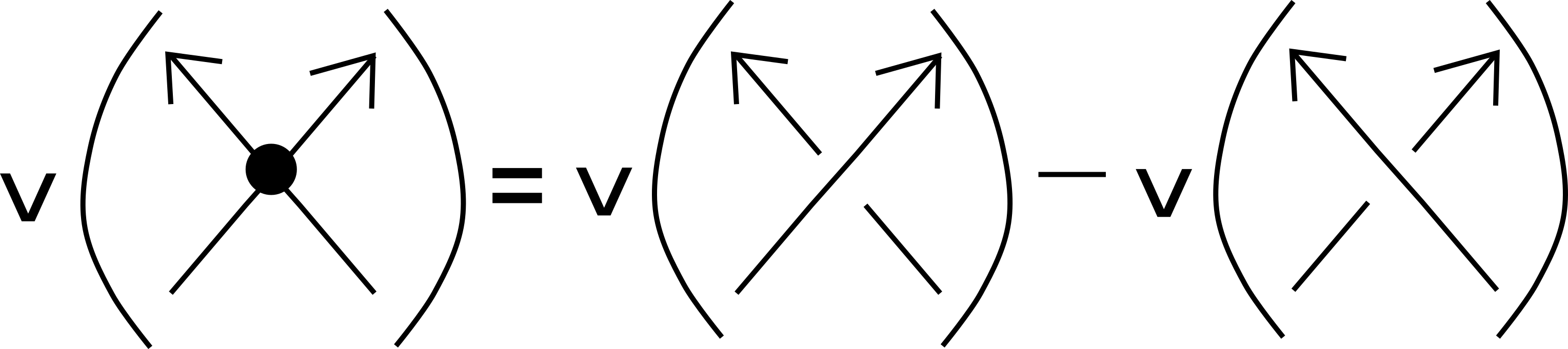}
\caption{How to resolve a double point.}
\label{doublepointresolution}
\end{figure}

This expression is also sometimes referred to as a ``formal derivative'' of the invariant $\nu$.
We say the invariant is \emph{of finite type} and \emph{has order $\leq n$} if it vanishes identically on all knots with more than $n$ double points (i.e. if its $n+1$st derivative vanishes).
The order of the invariant determines a filtration on the space of FTIs, so that we can define the space of invariants of order $=n$ as the quotient of the invariants of order $\leq n$ by those of order $\leq n-1$.
Many famous knot invariants are (or can be reduced to) finite-type invariants, for example the Jones and HOMFLY-PT polynomials (after a power series expansion).
Finite-type invariants for classical knots have been classified through the work of Vassiliev and Kontsevich, who showed that the space of Vassiliev invariants of order $=n$ is isomorphic to the dual space of the chord algebra on $n$ chords (denoted by $\mathcal{C}_n$), modulo the two relations in Figs. \ref{oneT} and \ref{4T} (see \cite{6}).

\begin{figure}[h]
\centering
\includegraphics[scale=.15]{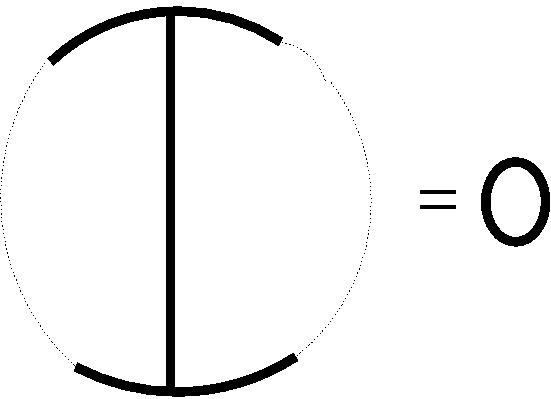}
\caption{The one-term relation.}
\label{oneT}
\end{figure}

\begin{figure}[h]
\centering
\includegraphics[scale=.15]{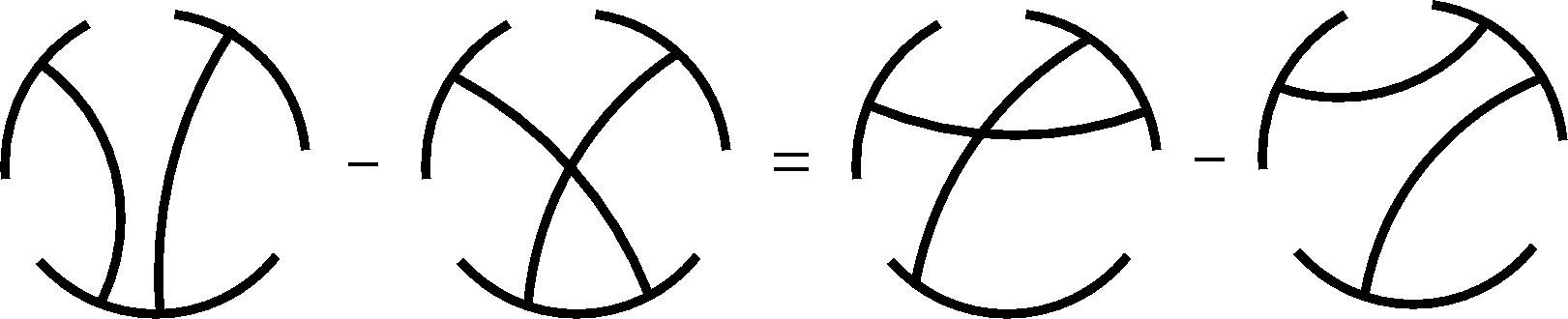}
\caption{The four-term relation, $\overline{4T}$.}
\label{4T}
\end{figure}
\noindent As usual in such pictures we assume that the diagrams coincide outside of the parts we drew.
We call Fig. \ref{oneT} the one-term relation and Fig. \ref{4T} the four-term relation ($\overline{4T}$) for chord diagrams.

For Kauffman \cite{3}, a finite-type invariant of virtual knots is the extension of a virtual knot invariant to the category of singular virtual knots.
These are virtual knots with an extra type of crossing (transverse double points), modulo the Reidemeister moves for virtual knots and some extra moves, which he calls \emph{rigid vertex isotopy} and that are represented in Fig. \ref{rigidvertexisotopy}. 

\begin{figure}[h]
\centering
\includegraphics[scale=.2]{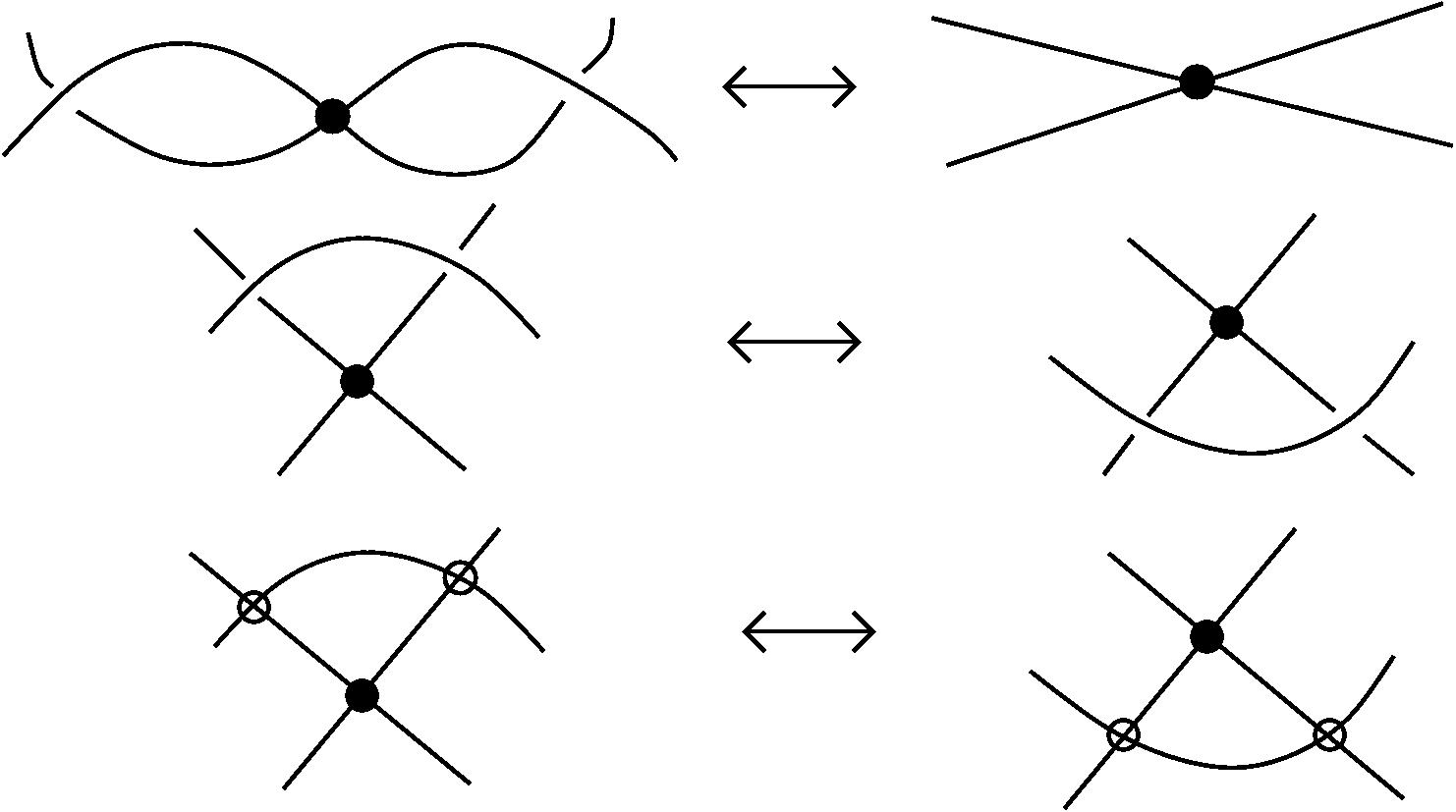}
\caption{The rigid vertex isotopy moves. }
\label{rigidvertexisotopy}
\end{figure}

\noindent The extension uses the same exact formula as in the classical case, see Fig. \ref{doublepointresolution}.
We then say that $\nu$ is a \emph{finite-type invariant of order $\leq n$} if it vanishes on every knot with more than $n$ double points (i.e. the $n+1$st derivative vanishes). 
Examples of finite-type invariants according to this definition are the coefficient of $x^n$ in the Conway polynomial or the Birman coefficients. 

A Vassiliev invariant is said to be \emph{universal} if any other invariant of the same order with values in some abelian group can be recovered from the universal invariant by some natural construction.
For example, a Vassiliev invariant of order one is the universal invariant of order one if we can recover the value of any other Vassiliev invariant $V$ of order $\leq 1$ on any knot $K$ using our universal invariant, the first derivative of $V$ and the value of $V$ on a representative for each of the homotopy classes of virtual knots with a double point.
The universal invariant $U_1$ from \cite{2} that we will later generalize satisfies the formula

\begin{equation}V(K)=V(K_0)+V^{(1)}(\frac{1}{2}(U_1(K)-U_1(K_0))),\end{equation}

\noindent where $K_0$ is in the homotopy class of $K$, $U_1$ is our universal invariant and $V^{(1)}$ is the first derivative of $V$ (where we're using the fact that $V^{(1)}$ is constant on homotopy classes of singular virtual knots with one double point).

\begin{remark}
There is another theory of finite-type invariants of virtual knots, due to Goussarov, Polyak and Viro \cite{1}.
We will not consider that theory in this paper, and focus instead on the theory of Vassiliev invariants as defined by Kauffman.
\end{remark}

\section{The Three Henrich Invariants and their Extensions}
\label{section3}

Throughout the section, we will first consider the extension of each invariant to framed, closed virtual knots, then the extensions to long virtual knots, both framed and unframed.

\subsection{The polynomial invariant}
\label{polyinv}
\begin{definition}\label{def1}
Pick a classical crossing $d$ of the framed virtual knot $K$ and smooth it according to orientation.
The result will be a two component virtual link; consider the flat virtual link associated to it.
Choose an arbitrary ordering $\{1,2\}$ of the components, and give a sign to every crossing in $1\cap 2$ according to the following rule:

\begin{figure}[!h]
\centering
\includegraphics[scale=.2]{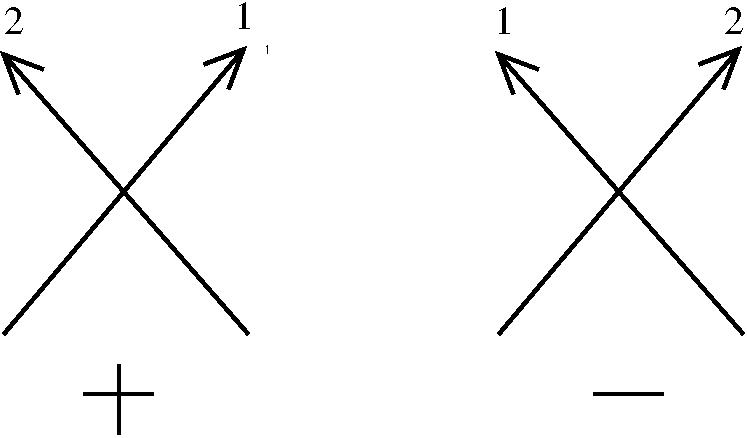}
\caption{The sign of a 2-component flat virtual link.}
\label{flatlinksign}
\end{figure}

The \emph{intersection index} of the crossing $d$ is the quantity given by 

\begin{equation}i(d)=\sum_{x\in 1\cap 2}sgn(x).\end{equation}
\end{definition}

\begin{lemma}
Let $L$ be an ordered two-component flat framed virtual link, with $1\cap 2$ the set of intersections of the two components.
Then $i(L)=\sum_{x\in1\cap 2}sgn(x)$ is invariant under framed flat equivalence.
\end{lemma}

It's an easy exercise to check that the intersection index is invariant under Reidemeister moves; in particular, flat framed Reidemeister move one only adds self-crossings, which don't influence the value of $i(L)$.
Moreover, changing the order of the components simply changes the sign of the intersection index.

\begin{definition}
Given a virtual knot $K$, its polynomial invariant $p_t(K)\in\Z[t]$ is given by the formula 
\begin{equation}p_t(K)=\sum_d sgn(d)t^{|i(d)|}\end{equation}

\noindent where $d$ is taken over all classical crossings of $K$.
\end{definition}

\begin{prop}\label{polyclosedframed}
The polynomial $p_t(K)$ is an order one Vassiliev invariant of framed virtual knots.
\end{prop}

\begin{proof}
The proof is basically the same as the one described in \cite{2} for the unframed case, and the same strategy will apply to all the invariants we study.
To verify that $p_t$ is an invariant, we must check invariance under the three framed classical Reidemeister moves as well as the mixed Reidemeister move.
Let's check that $p_t$ is invariant under the framed version of Reidemeister one.
Applying framed Reidemeister one introduces two new crossings of opposite sign.
Call those crossings ``$+$'' and ``$-$'' respectively; then we introduce two new terms in $p_t$, of the form 

\begin{equation}+t^{|i(+)|} -t^{|i(-)|}\label{eq1a}\end{equation}

\noindent However, it's easy to see that, after smoothing either of the crossings, the two components of the flat virtual link we get do not intersect.

\begin{figure}[!h]
\centering
\includegraphics[scale=.25]{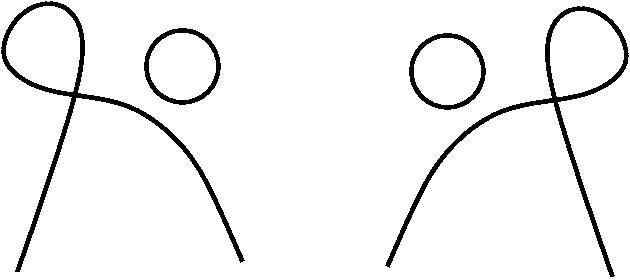}
\caption{The result of smoothing along either crossing created by a framed Reidemeister one move and taking the flat class.}
\label{figure1}
\end{figure}

\noindent Thus $i(+)=i(-)=0$ and the terms in Eq. (\ref{eq1a}) cancel each other, showing that $p_t$ is indeed invariant under framed Reidemeister one.
When introducing two new crossings with a classical Reidemeister move two, those crossings have opposite signs, and the results of smoothing along either crossing are clearly homotopic, so the overall contribution of a Reidemeister move two is zero.
For classical Reidemeister move three and the mixed Reidemeister move, it is easy to see that there is a one-to-one correspondence between the signs of the crossings and their intersection indices before and after the moves.
This shows that $p_t$ is an invariant of framed virtual knots.

To show that $p_t$ is an invariant of order one, we need to show that it vanishes on any knot with two double points, and has a nonzero value on a knot with one double point.
The first assertion is a consequence of the alternating sign: suppose you have a knot $K_{dd'}$ with two double points $\{d,d'\}$, and denote by $K_{ab}$, where $a,b\in\{+,-\}$, the four possible resolutions of the $d,d'$.
By definition, 
\begin{equation}\label{eqpt}
p_t(K_{dd'})=p_t(K_{++})-p_t(K_{+-})-p_t(K_{-+})+p_t(K_{--}).
\end{equation}
However, the intersection index is defined in terms of a smoothing, and oriented smoothings don't see the sign of the crossing, so $|i(d)|$ is the same whether $d$ is positively or negatively resolved.
By virtue of the alternating signs in Eq. (\ref{eqpt}) all the common terms cancel out, and the terms corresponding to $d$ (and similarly $d'$) each appear twice with opposite signs, so they cancel out as well.
The end result is that $p_t(K_{dd'})=0$.
Finally, $p_t$ is still nonzero on the example of \cite{2}, Fig. 12, where we glue one of the upper crossings.
\end{proof}

\begin{cor}
Let $p_t(mod2)$ be the polynomial invariant (for framed virtual knots) with coefficients in $\Z_2$ instead of $\Z$.
Then $p_t(mod2)$ is a homotopy invariant of framed virtual knots.
\end{cor}

\begin{proof}
$p_t$ is already invariant under classical and virtual Reidemeister moves (including framed Reidemeister one), so $p_t(mod2)$ also is.
Moreover, applying the CC move changes the coefficient $sgn(d)$ of the term associated to that crossing from $\pm1$ to $\mp1$.
Thus $p_t(mod2)$ is invariant under the CC move as well.
\end{proof}

$p_t(mod 2)$ doesn't play a huge part in what follows, but it is used in the proof that the smoothing invariant is strictly stronger than the polynomial invariant, so we needed to check that it was still an invariant in the framed case.

\subsection{The two polynomial invariants for long virtual knots}

\begin{definition}\label{def1long}
Pick a classical crossing $d$ of the knot $K$ and smooth it according to orientation.
The result will be a two component virtual tangle, with one closed component; consider the flat virtual tangle associated to it.
Choose an arbitrary ordering $\{1,2\}$ of the components, and give a sign to every crossing in $1\cap 2$ according to the rule in Fig. \ref{flatlinksign}.

\noindent The \emph{intersection index} of the crossing $d$ is given by 

\begin{equation}i(d)=\sum_{x\in 1\cap 2}sgn(x).\label{intersectionindex}\end{equation}
\end{definition}

\begin{lemma}
Let $L$ be an ordered two-component flat framed virtual tangle, with $1\cap 2$ the set of intersections of the two components.
Then $i(L)=\sum_{x\in1\cap 2}sgn(x)$ is invariant under flat equivalence, both in the framed and unframed case.
\end{lemma}

The proof of this lemma is an easy exercise.
As before, it is sufficient to check that the intersection index is invariant under Reidemeister moves.
Note that changing the order of the components simply changes the sign of the intersection index.

\begin{definition}
Given a long virtual knot $K$, its polynomial invariant $p_t(K)\in\Z[t]$ is given by the formula 

\begin{equation}p_t(K)=\begin{cases}\sum_d sgn(d)(t^{|i(d)|}-1)&\text{ in the unframed case }\\
\sum_d sgn(d)t^{|i(d)|} & \text{ in the framed case }\end{cases}\end{equation}
where $d$ is taken over all classical crossings of $K$ and $i(d)$ is the intersection index of Eq. (\ref{intersectionindex}).
\end{definition}

\begin{prop}\label{polylong}
The polynomial $p_t(K)$ is a Vassiliev invariant of order one of long virtual knots, both framed and unframed.
\end{prop}

\begin{proof}The proof is analogous to the one for Proposition \ref{polyclosedframed}.
In particular, the value of $p_t$ on the knot $K$ in Fig. \ref{ptnonzero} is nonzero, looking at it as both a framed and unframed long virtual knot.
In the unframed case, resolving the double point one way gives us the long unknot (whose $p_t=0$) and resolving the other way gives a knot whose polynomial invariant is $2(t^2-1)$, so overall $p_t(K)=2(t^2-1)$.
In the framed case resolving the double point we get the long unknot with one kink (whose $p_t=1$) and a knot whose polynomial invariant is $2t^2-1$, so overall we still have $p_t(K)=2t^2-2$.\end{proof}

\begin{figure}
\centering
\includegraphics[scale=.15]{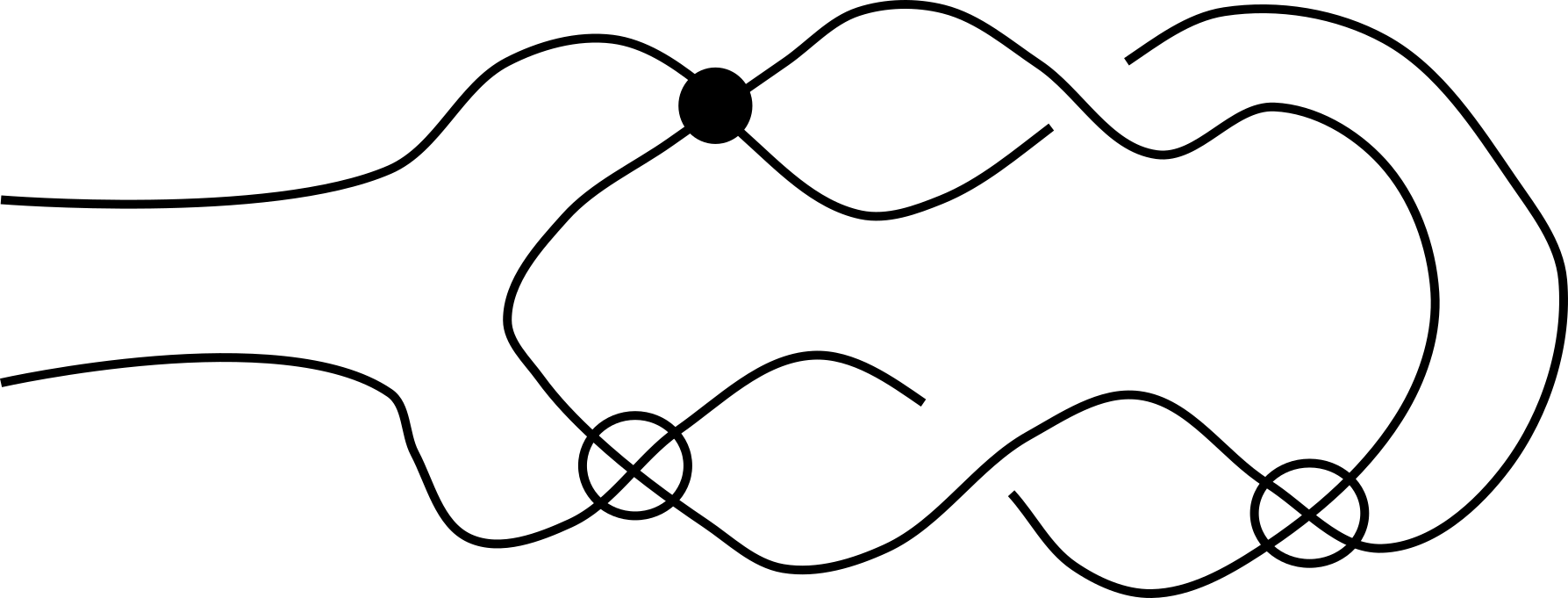}
\caption{A singular long knot on which $p_t\neq0$.}
\label{ptnonzero}
\end{figure}

\begin{figure}
\centering
\includegraphics[scale=.1]{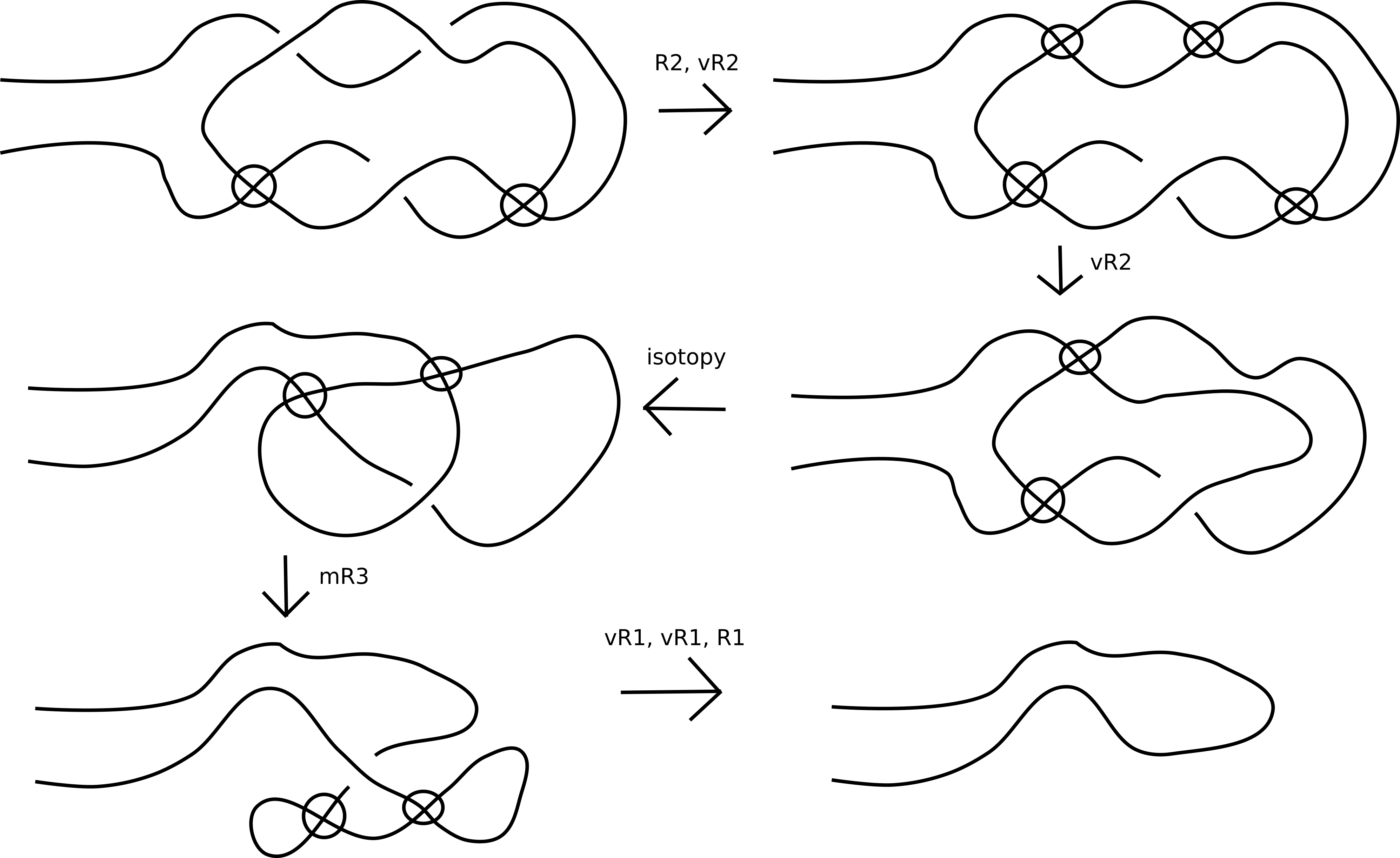}
\caption{How one resolution of Fig. \ref{ptnonzero} yields the (long) unknot in the unframed case. In the framed case only the very last step fails, so the knot is isotopic to the long knot with a single kink (whose $p_t$ value is $1$).}
\label{ptunknot}
\end{figure}

Now that we have established the invariant for long knots, we may investigate the effects of a few basic long knot moves (closure, connected sum) on the invariant.

\begin{theorem}\label{ptclosureinvariant}
The value of $p_t$ on long virtual knots only depends on the closure of the long virtual knot. 
Two long virtual knots with the same closure yield the same $p_t$ value.
\end{theorem}

\begin{proof}
To prove the statement, we will see how to compute the intersection index (and hence the polynomial invariant) from the Gauss diagram of the long knot.
This requires us to see what the effect of smoothing a crossing is in the Gauss diagram setting.
This is illustrated in Fig. \ref{smoothinggauss}.
To compute the intersection index of a crossing (as per Definition \ref{def1}) we must smooth the Gauss diagram at that crossing and look at the resulting virtual string (see Remark \ref{virtualstringdef}), which has two components.
Arbitrarily order those components and count the number of arrows that go from the first component to the second minus the number of arrows that go from the second component to the first.
It is easy to see that the absolute value of that number is $|i(d)|$.
Note that every arrow that has head and tails in the same component does not influence the value of $|i(d)|$, as they correspond to self-crossings of that component.

\begin{figure}[!h]
\centering
\includegraphics[scale=.1]{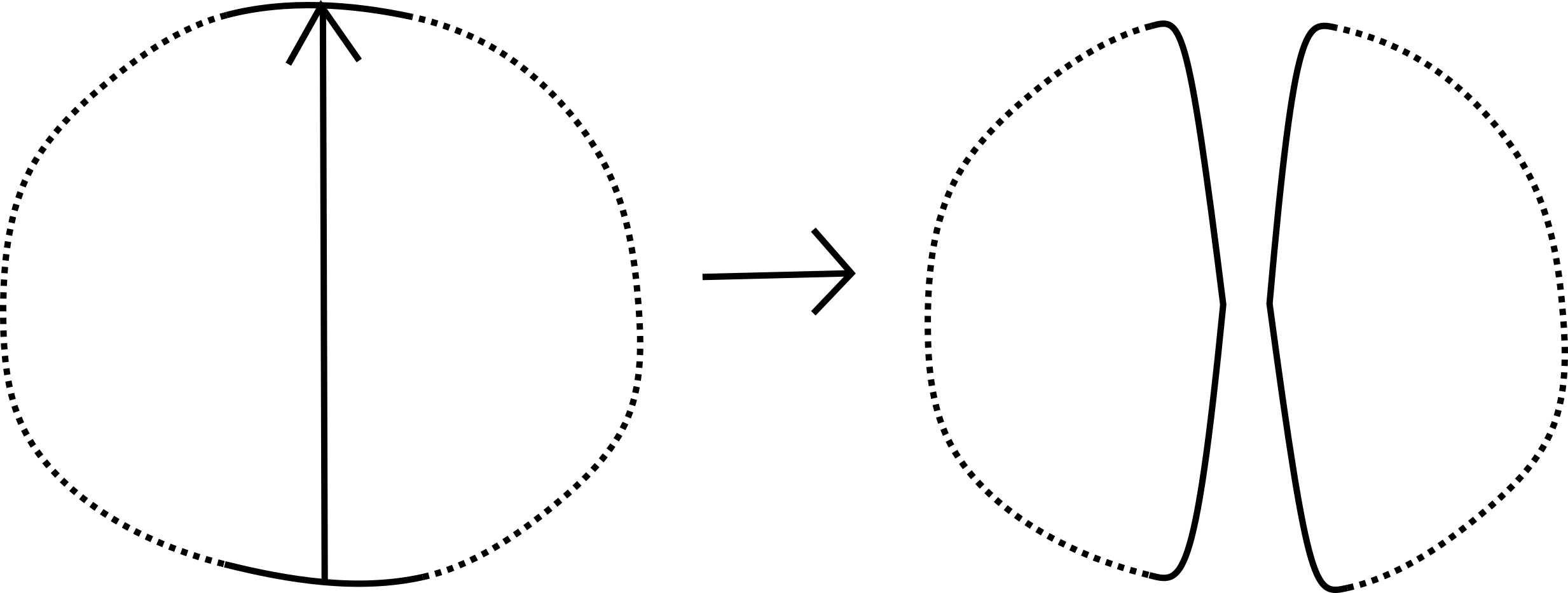}
\caption{The effect of smoothing a crossing on the Gauss diagram.}
\label{smoothinggauss}
\end{figure}

Now, to prove our assertion we need to show that when we move our fixed point past a crossing the intersection index doesn't change.
But that is the case, since the point at infinity never appeared in the previous description, and the only thing that can differ after moving the point at infinity past a crossing is which component is a long knot and which component is a closed knot.
In particular, the crossings between the two components do no change, so the intersection index remains the same, and the polynomial invariant of the knot does not vary.
Two long virtual knots with the same closure only differ by the positioning of the point at infinity, so by the above argument they will have the same polynomial invariant. 

\end{proof}

\begin{prop}
$p_t$ is additive with respect to the connected sum of long knots: $p_t(K_1\# K_2)=p_t(K_1)+p_t(K_2)$.
It is also independent on the order of the sum: $p_t(K_1\# K_2)=p_t(K_2\# K_1)$.
\end{prop}

\begin{proof}
Recall that, outside of a closed ball centered at the origin, the long knot is just a straight line. 
The connected sum of two long knots $K_1, K_2$ simply glues the end of $K_1$ to the beginning of $K_2$.
Because all of the ``knottedness'' of either knot is contained in the closed ball, whenever we smooth a crossing in $K_1$ the two components only intersect inside of the $K_1$ ball, and similarly for $K_2$.
Thus, the intersection index does not ``see'' the connected sum, and we can split the polynomial invariant in the following way:

\begin{equation}\begin{split}p_t(K_1\#K_2)=\sum_{v\in K_1\#K_2}(t^{|i(d)|}-1)&=\sum_{v\in K_1}(t^{|i(d)|}-1)+\sum_{v\in K_2}(t^{|i(d)|}-1)\\&=p_t(K_1)+p_t(K_2).\end{split}\end{equation}
The second assertion holds because the polynomial invariant only sees the closure of the long virtual knot and $K_1\#K_2$ has the same closure as $K_2\#K_1$.
\end{proof}

\begin{cor}
Let $p_t(mod 2)$ be the polynomial invariant (as defined above) with coefficients in $\Z_2$ instead of $\Z$.
Then $p_t(mod 2)$ is a homotopy invariant of long virtual knots.
\end{cor}

The motivation for the absolute value sign in the original polynomial invariant was our inability to distinguish the two components of the link we get after smoothing a classical crossing.
However, in the case of long virtual knots, the two components can be distinguished in a very natural way: one of them is a long knot, the other is a closed knot.
We can thus extend the polynomial invariant to the following, slightly more general, invariant.

\begin{definition}
Given a long virtual knot $K$, its \emph{\ul{ordered} polynomial invariant} $p_t(K)\in\Z[t]$ is given by the formula 

\begin{equation}\overline{p}_t(K)=\begin{cases}\sum_d sgn(d)(t^{i(d)}-1)&\text{ in the unframed case }\\
\sum_d sgn(d)t^{i(d)} & \text{ in the framed case }\end{cases}\end{equation}
where $d$ is taken over all classical crossings of $K$ and $i(d)$ is the intersection index of Definition \ref{def1} where we always pick the long component as component number one.
\end{definition}

\begin{prop}
The polynomial $\overline{p}_t(K)$ is a Vassiliev invariant of order one of long virtual knots, both framed and unframed.
\end{prop}

The proof is analogous to the one of Proposition \ref{polylong}, including the same example that $\overline{p}_t(K)\neq 0$, see Fig. \ref{ptnonzero}.
It is then worth asking whether the ordered polynomial invariant is stronger than the polynomial invariant.
As a reminder, invariant $A$ is stronger than invariant $B$ if whenever $A(K)=A(K')$ we also have $B(K)=B(K')$, but there is a pair $K_1, K_2$ such that $B(K_1)=B(K_2)$ and yet $A(K_1)\neq A(K_2)$.

\begin{prop}
The ordered polynomial invariant is stronger than the polynomial invariant.
\end{prop}

\begin{proof}
Since we get from the ordered polynomial invariant to the polynomial invariant by taking the absolute value of all the exponents, if two knots have the same ordered polynomial invariant they will also have the same polynomial invariant.
The two knots that show that the ordered polynomial is stronger than the polynomial invariant are pictured in Fig. \ref{k1k2smoothpoly}.

\begin{figure}[!h]
\centering
\includegraphics[scale=.115]{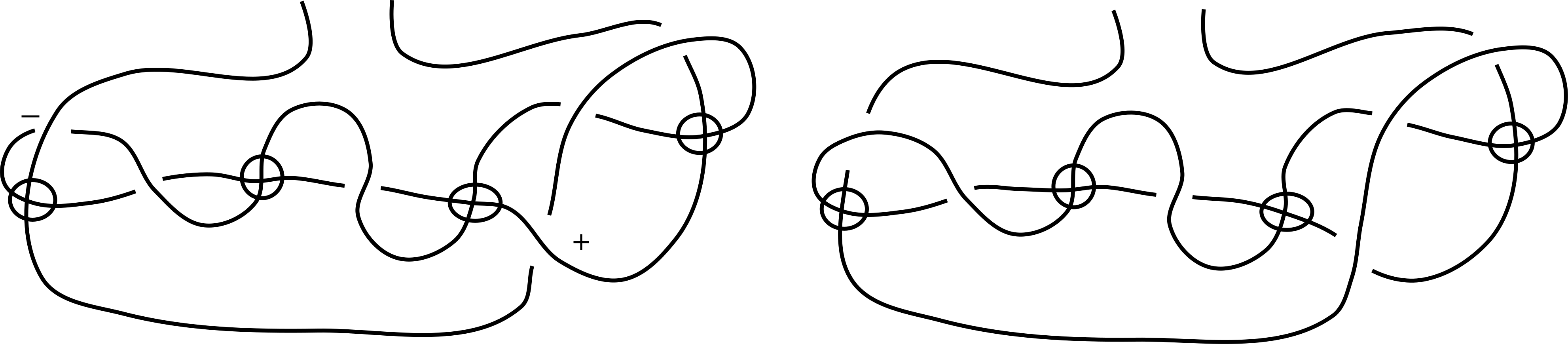}
\caption{Two long knots, $K_1$ (left) and $K_2$ (right), with the same $p_t$ value but different $\overline{p}_t$ value.}
\label{k1k2smoothpoly}
\end{figure}

These only differ at the crossings labeled $-$ and $+$ in $K_1$, which are switched in $K_2$, and these crossings have intersection index given by $i(-)=+2$ and $i(+)=-2$ (assuming we take the long component as component $1$).
Then in $p_t$, which looks at $|i(d)|$, their contributions will cancel out, and since the other crossings coincide we have $p_t(K_1)=p_t(K_2)$.

On the other hand, using the above observations, computing $\overline{p}_t(K_1)-\overline{p}_t(K_2)$ yields $2t^{-2}-2t^2$ (both in the framed and unframed case), so $\overline{p}_t(K_1)\neq \overline{p}_t(K_2)$.
\end{proof}

\begin{remark}
The statement on connected sum of long virtual knots still holds, so $\overline{p}_t(K_1\#K_2)=\overline{p}_t(K_1)+\overline{p}(K_2)$, and commutativity of addition in $\Z[t]$ still gives us that $\overline{p}_t(K_1\#K_2)=\overline{p}_t(K_2\#K_1)$. 
However, the lack of existence of an ordered polynomial for closed knots makes an analog of Proposition \ref{ptclosureinvariant} impossible.
In fact, it was precisely our ability to distinguish the two curves after the smoothing that led us to the definition of the ordered polynomial invariant. 
\end{remark}

\subsection{The smoothing invariant}
\label{section1}
\begin{definition}
Let $K$ be a virtual knot diagram.
Let $K^d_{sm}$ be the unordered two-component virtual link obtained by smoothing $K$ at a classical crossing $d$, and $[K^d_{sm}]$ its flat equivalence class.
Then the \emph{smoothing invariant}, which takes values in the free abelian group generated by homotopy classes of framed virtual links with two components, is given by the formula

\begin{equation}S(K)=\sum_d sgn(d)[K^d_{sm}]\end{equation}

\noindent where once again the sum is over all classical crossings of $K$.
\end{definition}

\begin{prop}
The smoothing invariant is an order one Vassiliev invariant of framed virtual knots.
\end{prop}

\begin{proof}
The proof is analogous to the one for the polynomial invariant, including the same example of knot such that $S(K)\neq 0$.
\end{proof}
\begin{prop}\label{smoothingstronger}
The smoothing invariant is stronger than the polynomial invariant
\end{prop}

\begin{proof}
Recall that the statement means that if two knots have the same $S$ value they have the same $p_t$ value, and that there are two knots $K_1, K_2$ such that $p_t(K_1)=p_t(K_2)$ but $S(K_1)\neq S(K_2)$.

Since the intersection index can be recovered from $[K^d_{sm}]$, if two knots have the same smoothing invariant they must also have the same polynomial invariant (because they have the same collection of $[K^d_{sm}]$, which determines $|i(d)|$, which determines $p_t$).
The two knots with $p_t(K_1)=p_t(K_2)$ but $S(K_1)\neq S(K_2)$ are pictured in Fig. \ref{smoothingandpoly}.

\begin{figure}[!h]
\centering
\includegraphics[scale=.17]{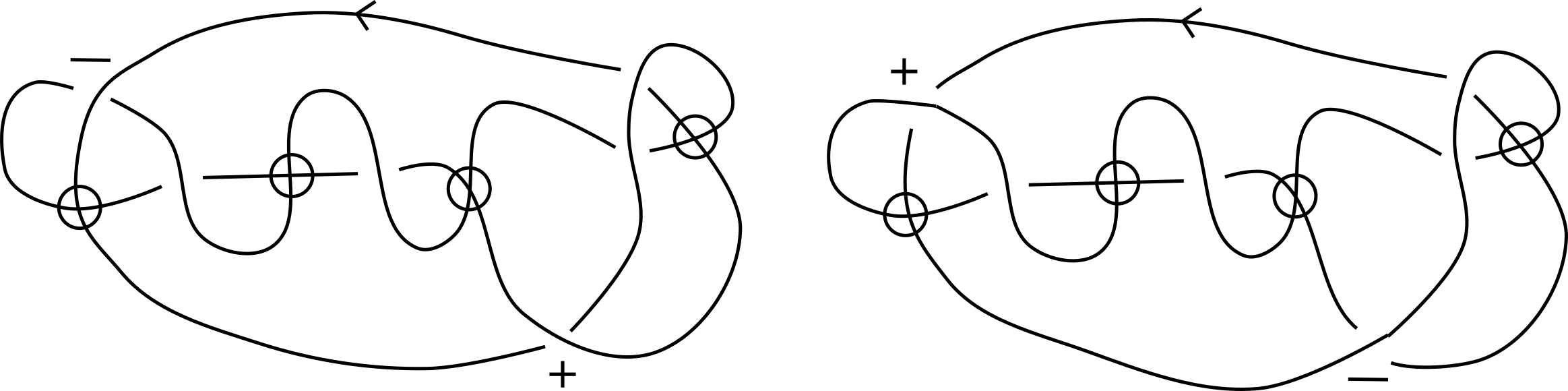}
\caption{The two knots that show that the smoothing invariant is stronger than the polynomial invariant.}
\label{smoothingandpoly}
\end{figure}

It is easy to see that these two knots simply differ by a switching of the $+$ and $-$ crossings in the picture.
Because both crossings have the same $|i(d)|$ value and have opposite sign, the overall contribution to the value of $p_t$ is zero.
This shows that $p_t(K_1)=p_t(K_2)$.
Computing $S(K_1)-S(K_2)$ we get $2([K_{1+}]-[K_{1-}])$, where $K_{1+}$ is the flat virtual link obtained by smoothing at the $+$ crossing in Fig. \ref{smoothingandpoly},
and similarly for $K_{1-}$.
To distinguish these two flat long knot classes we will use another invariant, closely related to the Goldman-Lie bracket.
This is an invariant of ordered two component virtual links, and is defined by the formula
\begin{equation}B(K)=\sum_{x\in 1\cap 2}sgn(x)[K_x],\end{equation}
where $x$ is a crossing that involves both components and $[K_x]$ is the flat class of the smoothing of $K$ along $x$.
Invariance under Reidemeister moves is an easy check, as any cases where the move only involves one of the components do not contribute to $B$.
Computing $B$ gives us a linear combination of flat virtual knots (since smoothing a crossing involving two components will join them).

To distinguish $B(K_{1+})$ from $B(K_{1-})$ we compose them with $p_t (mod 2)$; we get that $p_t(mod 2)(B(K_{1+}))=0$ and $p_t(mod 2)(B(K_{1-}))=t^3+t$.
Since they have distinct $p_t(mod 2)$ value, $B(K_{1+})\neq B(K_{1-})$, so $[K_{1+}]\neq [K_{1-}]$, so $S(K_1)\neq S(K_2)$.

\end{proof}

\begin{definition}
Let $K$ be a long virtual knot diagram.
Let $K^d_{sm}$ be the two-component virtual tangle obtained by smoothing $K$ at a classical crossing $d$, and $[K^d_{sm}]$ its flat equivalence class.
Denote by $[K^0_{link}]$ the flat equivalence class of the union of $K$ with an unlinked copy of the unknot.
Then the smoothing invariant is given by the formula

\begin{equation}S(K)=\begin{cases}
\sum_d sgn(d)([K^d_{sm}]-[K^0_{link}]) & \text{ in the unframed case}\\
\sum_d sgn(d)[K^d_{sm}] & \text{ in the framed case}
\end{cases}\end{equation}
where once again the sum is over all classical crossings of $K$.
\end{definition}

\begin{prop}
The smoothing invariant is a Vassiliev invariant of order one for long virtual knots, both framed and unframed.
\end{prop}

\begin{proof}
The proof is analogous to the one for the polynomial invariant.
\end{proof}

\begin{prop}
The smoothing invariant is strictly stronger than the ordered polynomial invariant.
\end{prop}

\begin{proof}
Note that, since the ordered polynomial invariant is stronger than the polynomial invariant, this result implies that the smoothing invariant is stronger than the polynomial invariant for long virtual knots as well.
The proof proceeds the exact same way as the one for the closed case, with the pair of knots in Fig. \ref{k1k2orderedsmoothing} being $K_1, K_2$.

\begin{figure}[!h]
\centering
\includegraphics[scale=.08]{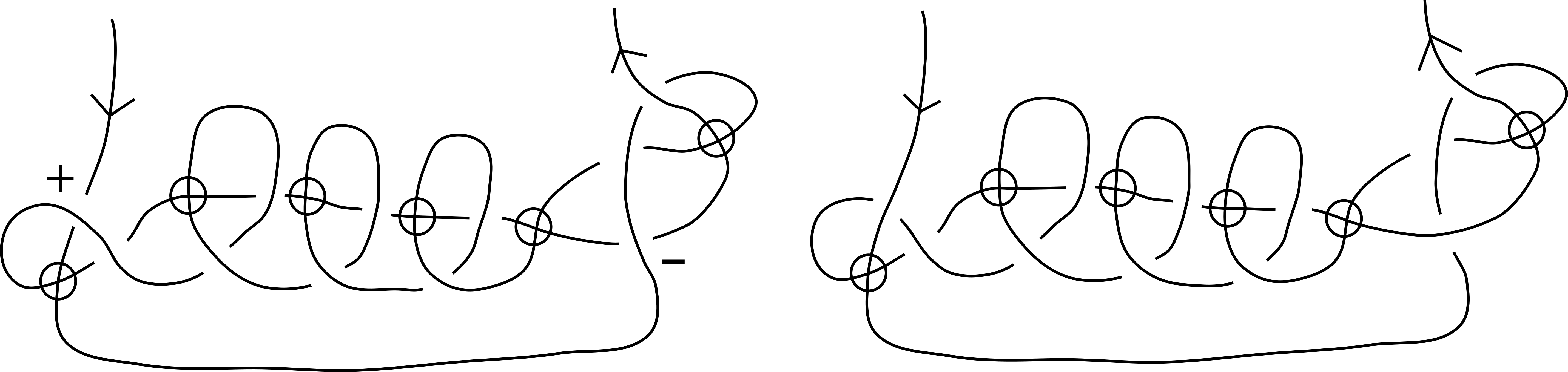}
\caption{The two knots, $K_1$ on the left and $K_2$ on the right, that show that the smoothing invariant is stronger than the ordered polynomial invariant.}
\label{k1k2orderedsmoothing}
\end{figure}

This pair of knots only differs at the crossings marked $+,-$ in $K_1$ which are, as in previous examples, switched in $K_2$.
We have, picking as usual the long component as component $1$, $i(+)=i(-)=-2$, and all the other crossings are the same, so that $\overline{p}_t(K_1)=\overline{p}_t(K_2)$.
On the other hand, computing $S(K_1)-S(K_2)$ yields $2([K_{1+}]-[K_{1-}])$, where $K_{1+}$ is the tangle obtained by smoothing $K_1$ at the crossing $+$, $K_{1-}$ is the tangle obtained by smoothing $K_1$ at the crossing $-$ and the square brackets represent, as usual, the flat class.
We then compute $B([K_{1+}])$ and $B[K_{1-}]$, and then apply $p_t(mod 2)$ to the result, like in the proof of Proposition \ref{smoothingstronger}.

The result of such computations gives $p_t(mod2)(B([K_{1+})])=0$, while $p_t(mod2)(B([K_{1-}]))=t^2$.
As before, this means $p_t(mod2)$ distinguishes $B([K_{1+}])$ and $B([K_{1-}])$, so $[K_{1+}]\neq[K_{1-}]$ and consequently $S(K_1)\neq S(K_2)$.
\end{proof}

\subsection{The glueing invariant}
\begin{definition}
Let $K$ be a virtual knot diagram.
Let $[K^d_{glue}]$ be the flat equivalence class of the singular knot obtained by glueing the crossing $d$ of $K$ into a double point.
Then \emph{the glueing invariant} $G(K)$ is given by the following formula (where, as before, the sum ranges over all classical crossings of $K$):

\begin{equation}G(K)=\sum_d sgn(d)[K^d_{glue}].\end{equation}
\end{definition}

\begin{theorem}
The glueing invariant is the universal order one Vassiliev invariant of framed virtual knots.
\end{theorem}

\begin{proof}
The proof that it's an invariant or that it's a Vassiliev invariant of order one is similar to the previous ones.
In particular, invariance under framed Reidemeister one comes from the fact that the two crossings we introduce have opposite sign, and the corresponding two terms in $G(K)$ are homotopic, since we can move a little kink with a double point around the knot using Reidemeister moves and the rigid vertex isotopy moves.

Universality (as defined in section \ref{finitetypeinvariants}) comes from the same argument of \cite{2}, which is as follows.
Let $K, K_0$ two long knots in the same homotopy class.
By definition, they are related by a sequence of Reidemeister moves and CC moves.
Let $K_0, K_1, \ldots, K_m=K$ be a sequence of knots where $K_i$ differs from $K_{i-1}$ by the application of some Reidemeister moves and exactly one CC move, and let $V$ be any Vassiliev invariant of order one.
Then using a telescoping sum trick we have the following:

\begin{equation}
\begin{split}
V(K)-V(K_0) &= V(K_m)-V(K_{m-1})+V(K_{m-1})-\cdots-V(K_1)+V(K_1)-V(K_0)\\
&=\sum_{i=1}^mV(K_i)-V(K_{i-1})\\
&=\sum_{i=1}^m sgn(i)V^{(1)}([K_{i\bullet}])
\end{split}
\end{equation}
Here $K_{i\bullet}$ is obtained by glueing in $K_i$ the crossing that we switch going from $K_{i-1}$ to $K_i$.
We can take the flat class $[K_{i\bullet}]$ because, by definition of finite-type invariant of order one, $V^{(1)}$ is constant on homotopy classes of knots with one double point.
If resolving positively the double point we get $K_i$ set $sgn(i)=+1$, if resolving positively the double point we get $K_{i-1}$ set $sgn(i)=-1$.
Then 
\begin{equation}
\begin{split}
V(K)-V(K_0) &= \sum_{i=1}^m sgn(i)V^{(1)}([K_{i\bullet}])\\
&=V^{(1)}(\sum_{i=1}^msgn(i)[K_{i\bullet}])\\
&=V^{(1)}\bigl(\frac{1}{2}(G(K)-G(K_0))\bigr).
\end{split}
\end{equation}

\end{proof}

\begin{definition}
\label{glueinginvariant}
Let $K$ be a long virtual knot diagram.
Let $[K^d_{glue}]$ be the flat equivalence class of the singular knot obtained by glueing the crossing $d$ of $K$ into a double point.
Let $[K^0_{sing}]$ be the flat equivalence class of $K$ with an added double point, introduced via a Reidemeister one move and glueing the newly-created crossing as above.
Note that the flat equivalence class of $[K^0_{sing}]$ does not depend on where we introduce the kink, as we can move the double point around using the CC move and the moves in Fig. \ref{rigidvertexisotopy}.
Finally $G(K)$ is given by the following formula (where, as above, the sum ranges over all classical crossings of $K$):

\begin{equation}
G(K)=\begin{cases}
\sum_d sgn(d)([K^d_{glue}]-[K^0_{sing}])& \text{ in the unframed case}\\
\sum_d sgn(d)[K^d_{glue}]& \text{ in the framed case}
\end{cases}\end{equation}
\end{definition}

\begin{theorem}
The glueing invariant is the universal order one Vassiliev invariant of long virtual knots, both framed and unframed.
\end{theorem}

\begin{proof}
The proof is analogous to the one for the closed case.
\end{proof}

\begin{theorem}\label{thm1}
The glueing invariant is stronger than the smoothing invariant, for framed/unframed, long/closed virtual knots.
\end{theorem}
Universality of $G$ implies that, if $G(K_1)=G(K_2)$, we also have $S(K_1)=S(K_2)$.
To show that $G$ stonger than $S$ we still need to exhibit a pair of long virtual knots such that $S(K_1)=S(K_2)$ but $G(K_1)\neq G(K_2)$, and prove these two statements.
To do so we will need to extend the theory of based matrices for virtual strings and singular virtual strings to the long case.
We will prove Theorem \ref{thm1} in section \ref{universality}.

\section{Virtual Strings and Based Matrices}
\label{section4}

The goal of this section is to introduce and extend the notion of a based matrix of a virtual string.
This invariant of virtual strings (i.e. flat virtual knots) was first introduced by Turaev \cite{5}, and later generalized by Henrich \cite{2} for flat virtual knots with one double point.
We will extend the invariants to the categories of framed, closed virtual knots and of long virtual knots, both framed and unframed.
We will then use these invariants to prove Theorem \ref{thm1} in all the cases we considered.

\subsection{The based matrix of a virtual string}
\label{turaev}

Before generalizing the notion of based matrix of a virtual string to the framed case, we will recall the construction as it was defined in \cite{5} and \cite{2}.

\begin{definition}
For an integer $m>0$, a \emph{virtual string} $\alpha$ of rank $m$ is an oriented circle $S$, called the core circle of $\alpha$, and a distinguished set of $2m$ distinct points on the circle partitioned in ordered pairs. 
We call these pairs the arrows of $\alpha$; the collection of all arrows is denoted $arr(\alpha)$. 
The endpoints of an arrow $(a,b)\in arr(\alpha)$ are called respectively the tail ($a$) and the head ($b$) of the arrow. 
Finally, two virtual strings are homeomorphic if there is an orientation-preserving homeomorphism of the core circles sending the arrows of the first string to the arrows of the second. We only consider virtual strings up to homeomorphism.
 \end{definition}

Drawing the flat virtual knot diagram associated  to the virtual string is the same as drawing the knot associated to a Gauss diagram. 
More rigorously, every crossing of a flat virtual knot diagram corresponds to an arrow in the associated virtual string as illustrated by Fig. 18.

\begin{figure}[!h]
\centering
\includegraphics[scale=.2]{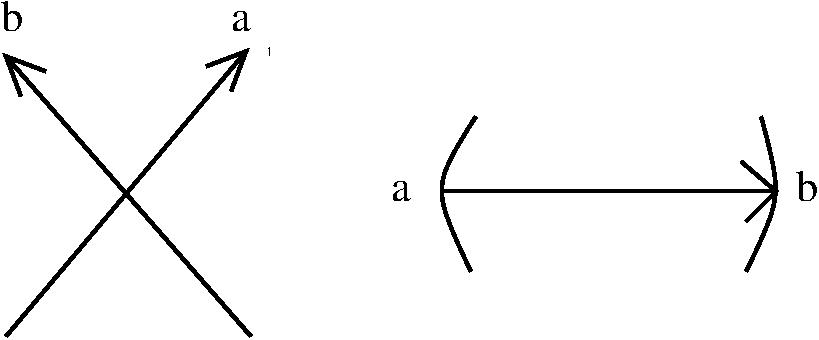}
\caption{A crossing and the corresponding arrow in its virtual string.}
\end{figure}

\begin{definition}\label{def2}
A \emph{based matrix} is a triple $(G, s, b\colon G\times G\to H)$, where $H$ is an abelian group, $G$ is a finite set, $s\in G$ and the map $b$ is skew-symmetric.
\begin{enumerate}
\item We call $g\in G\setminus \{s\}$ annihilating if $b(g,h)=0$ for all $h\in G$;
\item We call $g\in G\setminus \{s\}$ a core element if $b(g,h)=b(s,h)$ for all $h\in G$;
\item We call $g_1,g_2\in G\setminus \{s\}$ complementary if $b(g_1,h)+b(g_2,h)=b(s,h)$ for all $h\in G$.
\end{enumerate}
\end{definition}

\noindent Clearly the term ``based matrix'' comes from the matrix representation of $b$.
We always choose the core element $s$ to be the first element in the basis of the matrix representation; the first row will then show the values $b(s,e)$.
We can use the three element types to define \emph{elementary extensions} of a based matrix:

\begin{itemize}
\item $M_1$ changes $(G,s,b)$ into $(G\coprod \{g\}, s, b_1)$ such that $b_1$ extends $b$ and $g$ is annihilating;
\item $M_2$ changes $(G,s,b)$ into $(G\coprod \{g\}, s, b_2)$ such that $b_2$ extends $b$ and $g$ is a core element;
\item $M_3$ changes $(G,s,b)$ into $(G\coprod \{g_1,g_2\}, s, b_3)$ such that $b_3$ is any skew-symmetric map extending $b$ in which $g_1, g_2$ are complementary.
\end{itemize}

\begin{definition}\label{BMiso}
A based matrix is called \emph{primitive} if it cannot be obtained from another matrix by elementary extensions.
Two based matrices $(G,s,b)$ and $(G', s', b')$ are \emph{isomorphic} if there is a bijection $G\to G'$ mapping $s\mapsto s'$ and transforming $b$ into $b'$.
Two based matrices are \emph{homologous} if they can be obtained from each other by a finite number of elementary extensions and their inverses. 
\end{definition}

\begin{lemma}[\cite{5}]
Every based matrix is obtained from a primitive based matrix by elementary extensions.
Two homologous primitive based matrices are isomorphic.
\end{lemma}

This lemma is useful because it lets us extend an isomorphism invariant of based matrices to a homology invariant.
We do so by taking a primitive based matrix in the homology class and evaluating the invariant on it; any other primitive based matrix in the same homology class is isomorphic to it, so the invariant is well-defined.
A simple invariant of the sort is the cardinality of the finite set $G$.

To associate a based matrix to a virtual string, suppose you have the virtual string $\alpha$.
Let $G=G(\alpha)$ be the set $\{s\}\coprod arr(\alpha)$.
Then we can define a skew-symmetric map $b=b(\alpha)\colon G\times G\to \Z$ using intersection indices of some curves obtained from $\alpha$.
To compute $b(e,s)$ for $e\in arr(\alpha)$ we smooth the double point corresponding to $e$ and compute the intersection index $i(e)$ (see definition \ref{def1}) of the right-hand curve obtained from the smoothing with the left-hand curve.

\begin{figure}[h]
\centering
\includegraphics[scale=.3]{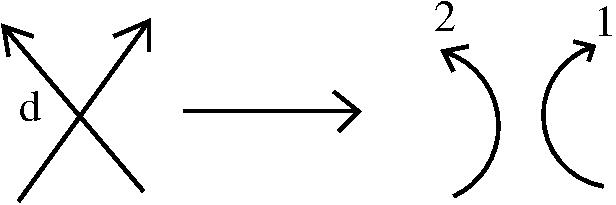}
\caption{Computing $b(d,s)$}
\end{figure}

If $e,f\neq s$, let $e=(a,b), f=(c,d)$ and let $(xy)^\circ$ denote the interior of the arc $xy$.
Count the number of arrows with tails in $(ab)^\circ$ and heads in $(cd)^\circ$ minus the number of arrows with heads in $(ab)^\circ$ and tails in $(cd)^\circ$.
Call this integer $ab\cdot cd$; then $b(e,f)=ab\cdot cd+\epsilon$, where $\epsilon\in\{0,\pm1\}$ according to the rule in Fig. \ref{linkingarrows}.

\begin{figure}[h]
\centering
\includegraphics[scale=.25]{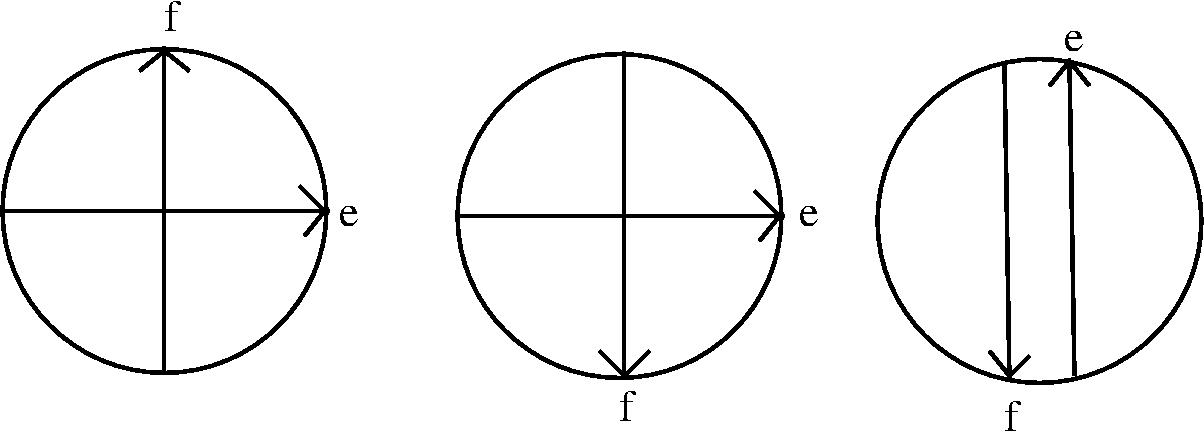}
\caption{In computing $b(e,f)$, $f$ links $e$ positively ($\epsilon=1$), $f$ links $e$ negatively ($\epsilon=-1$) and $e,f$ are unlinked ($\epsilon=0$).}
\label{linkingarrows}
\end{figure}

\noindent Then $(G(\alpha), s, b(\alpha))$ is a based matrix corresponding to the virtual string $\alpha$.

\begin{theorem}
If two virtual strings are homotopic, their based matrices are homologous.
\end{theorem}

\begin{definition}
A singular based matrix (SBM) is a quadruple $(G,s,d,b\colon G\times G\to H)$ where $G$ is finite, $s,d\in G$ and $b$ is skew-symmetric.
\end{definition}

The element types as in definition \ref{def2} change is the following way:

\begin{enumerate}
\item We call $g\in G\setminus \{s,d\}$ annihilating if $b(g,h)=0$ for all $h\in G$;
\item We call $g\in G\setminus \{s,d\}$ a core element if $b(g,h)=b(s,h)$ for all $h\in G$;
\item We call $g_1,g_2\in G\setminus \{s,d\}$ complementary if $b(g_1,h)+b(g_2,h)=b(s,h)$ for all $h\in G$.
\item We call $g\in\{s,d\}$ annihilating-like if $b(g,h)=0$ for all $h\in G$.
\item We say $d$ is core-like if $b(s,h)=b(d,h)$ for all $h\in G$.
\end{enumerate}

We define elementary extensions of SBM exactly as in the non-singular case.
There is a fourth elementary operation, called a \emph{singularity switch}:
\begin{itemize}
\item Suppose that there is a $g\in G$ such that $b(g,h)+b(d,h)=b(s,h)$ for all $h\in G$.
Then $N$ changes $(G,s,d,b)$ into $(G,s,g,b)$.
\end{itemize}

\begin{definition}\label{SBMiso}
Two SBMs  $(G,s,d,b)$ and $(G', s',d', b')$ are \emph{isomorphic} if there is a bijection $G\to G'$ such that $s\mapsto s'$, $d\mapsto d'$ and $b$ becomes $b'$.
A SBM is \emph{primitive} if it cannot be obtained from another SBM by elementary extensions, even after applying the singularity switch operation.
Two SBMs are \emph{homologous} if they can be obtained from each other by applying a finite number of elementary extensions, their inverses and singularity switches.
\end{definition}

One of the differences of the theory of SBMs is that every homology class either has a unique primitive SBM, or there are a pair of primitive SBMs that differ by the choice of the $d$ element.

To associate a SBM to any singular flat virtual knot with one double point, we compute the based matrix as if the double point was a regular flat crossing, and identify the distinguished element $d$ with the double point.
In the matrix representation of $b$ we always write $d$ in the last row/column, similarly to how $s$ is always the first row/column.
Then, at a diagram level, the reason why there could be a pair of primitive SBMs in a single homology class is that the double point might be a ``Reidemeister one''-type kink or might be in a ``Reidemeister two''-type pair.

\begin{theorem}
If two singular virtual strings are homotopic, their corresponding SBM are homologous.
\end{theorem}

\subsection{The framed case}
\label{framedmatricessection}
With the previous section in mind, let's generalize the SBM construction to the framed flat virtual knot case.
The first step is noticing that Reidemeister move one gave us both the annihilating and the core elements in the based matrix definition.
Thus we need to replace both of those definition by something suitable in the framed case.

\begin{definition}\label{elemmovesframed}
The elementary extensions for based matrices of framed virtual strings are the following.
\begin{itemize}
\item $\hat{M}_1$ changes $(G,s,b)$ into $(G\coprod \{g_1,g_2\}, s, b_1)$ such that $b_1$ extends $b$ and $g_1, g_2$ are annihilating;
\item $\hat{M}_2$ changes $(G,s,b)$ into $(G\coprod \{g_1,g_2\}, s, b_2)$ such that $b_2$ extends $b$ and $g_1, g_2$ are core elements;
\item $\hat{M}_3$ changes $(G,s,b)$ into $(G\coprod \{g_1,g_2\}, s, b_3)$ such that $b_3$ is any skew-symmetric map extending $b$ in which $g_1, g_2$ are complementary.
\item $\hat{M}_w$ changes $(G,s,b)$ into $(G,s,b')$ by acting on an annihilating element $g$ and changing it into a core element.
\end{itemize}
\end{definition}

\begin{remark}
A couple of things to note about the above definition:
\begin{itemize}
\item $\hat{M}_3$ is unchanged from the unframed case. 
We simply renamed it for consistency of notation. 
\item The move $\hat{M}_w$ is inspired by the Whitney trick, as in Fig. \ref{figure3}.
It's easy to see that $\hat{M}_w\sim\hat{M}_1^{-1}\circ \hat{M}_3$, and that $\hat{M}_w^{-1}\sim \hat{M}_2^{-1}\circ \hat{M}_3$, where we pick $g_1,g_2$ in $\hat{M}_3$ to be a pair \{annihilating, core\}. 
This $\hat{M}_3$ move is how the Whitney trick translates to virtual strings.
While introducing $\hat{M}_w$ increases the complexity of the upcoming proofs, it will turn out that $\hat{M}_w$ captures an indeterminacy that we cannot get rid of.

\begin{figure}[h]
\centering
\includegraphics[scale=.3]{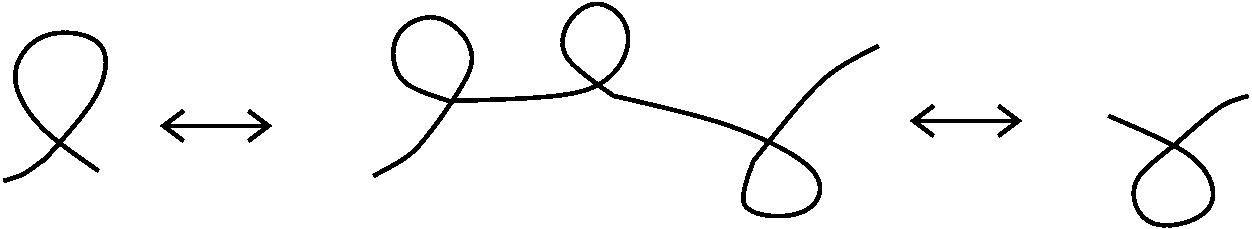}
\caption{The idea behind the move $\hat{M}_w$: use the Whitney trick to change which side the kink is on.}
\label{figure3}
\end{figure}

\item The notation $\sim$ means ``up to isomorphism''.
\item The motivation for our definition of $\hat{M}_1, \hat{M}_2$ can be found in \cite{5}, sections $4.2$ and $7.1$, where Turaev explains the origin of the moves $M_1, M_2$ in terms of Reidemeister move one.
We will come back to this motivation in the next section, when developing the theory of long virtual strings.
\end{itemize}
\end{remark}

With this new definition of elementary extensions, we can talk about primitive, isomorphic and homologous based matrices; these notions are analogous to the unframed case, see Definition \ref{BMiso}.

\begin{lemma}[see \cite{5}, Lemma 6.2.1]\label{lemma1}
Every based matrix is obtained from a primitive based matrix by elementary extensions.
Two homologous primitive based matrices are isomorphic up to a single application of $\hat{M}_w$.
\end{lemma}

\begin{proof}
For the first statement, starting with a based matrix we apply $\hat{M}_i^{-1}$, $i=1,2,3$, until we can't anymore, then try to apply $\hat{M}_w^{\pm 1}$ and see if we can apply even more inverse elementary extensions.
The cardinality of $G$ is monotone decreasing under this process, so it will at some point terminate, and the resulting matrix will be primitive by definition.

The proof of the second statement is similar in spirit to the one in \cite{5}.
It relies heavily on the following claim:

\begin{claim}
\label{claim1}
A move $\hat{M}_i$ followed by a move $\hat{M}_j^{-1}$ ($i,j\in \{1,2,3,w\}$) has the same effect as an isomorphism, a move $\hat{M}_k^{\pm1}$ or a move $\hat{M}_k^{-1}$ followed by a move $\hat{M}_l$.
Moreover, $\hat{M}_w\circ\hat{M}_j\sim\hat{M}_j\circ\hat{M}_w$ or $\hat{M}_w\circ\hat{M}_j\sim\hat{M}_k$ for $j,k\neq w$.
\end{claim}

Assume the claim holds.
Suppose that two primitive matrices $T,T'$ are homologous, i.e. they're related by a finite sequence of $\hat{M}_i^{\pm1}$ and isomorphisms.
Clearly isomorphisms commute with elementary extensions, so we can just stack them at the end of the sequence.
Using the first statement of the claim we can then rewrite the sequence of moves in the following way:

\begin{equation}\label{eq1}\hat{M}_{i_1}\cdots \hat{M}_{i_n}\hat{M}^{-1}_{j_1}\cdots\hat{M}^{-1}_{j_m}T\sim T'\end{equation}

\noindent Moreover, using the second statement of the claim we can stack the terms that look like $\hat{M}_w^{\pm1}$ in the middle of the sequence, so we can rewrite Eq. (\ref{eq1}) in the following way:

\begin{equation}\hat{M}_{i_1}\cdots \hat{M}_{i_n}\hat{M}_w^k\hat{M}^{-1}_{j_1}\cdots\hat{M}^{-1}_{j_m}T\sim T';\end{equation}

\noindent where $i_l, j_l\neq w$ for all $l$, and $k\in\Z$.
By primitivity of $T,T'$ we cannot apply $\hat{M}_j^{-1}$ to $T$ or obtain $T'$ via $\hat{M}_i$, so the only thing that survives is

\begin{equation}\hat{M}_w^kT\sim T'.\end{equation}
Finally, if $|k|>1$ we could actually cancel an even power out of it, since we can rewrite $\hat{M}_w^2\sim \hat{M}_2\circ \hat{M}_1^{-1}$, $\hat{M}_w^{-2}\sim \hat{M}_1\circ\hat{M}_2^{-1}$, and delete the terms we obtained as we did above.
This shows that $|k|\leq 1$ and proves the second statement of the lemma.
\end{proof}

\begin{remark}
The final part of the proof simply shows that, since we can now only delete a pair of rows of zeroes/core elements in our matrix, it is possible that we are left with a single annihilating/core row (equivalent on the knot to a single kink), and the $\hat{M}_w$ move changes the annihilating row into a core row or vice versa (changes the rotation of the kink).
$|k|\leq 1$ means can only end up with at most one kink, as if we had two we could slide them along the knot until they're next to each other and use a sequence of framed flat Reidemeister one moves to delete them.
\end{remark}

\begin{proof}[Proof of Claim \ref{claim1}]
Let's start with the first statement.
We need to check this for all possible choices of $i,j$.

\begin{itemize}
\item $i,j\in\{1,2\}$: We first add two elements $\{g_1.g_2\}$, then remove two elements $\{h_1,h_2\}$. If $i=j$ either these elements coincide, in which case $M_j^{-1}\circ M_i=id$, or they're disjoint, in which case the two operations commute, or one of them coincides, e.g. $g_1=h_1$, in which case $\hat{M}_j^{-1}\circ\hat{M}_i$ is an isomorphism (we replaced $h_2$ in the original matrix with $g_2$, another element of the same type).
If $i\neq j$ the same considerations as above hold: either $\{g_1,g_2\}$ is disjoint from $\{h_1,h_2\}$, in which case the two operations commute, or they coincide in at least one element, say $g_1=h_1$.
But since $i\neq j$ this means that $g_1$ is both annihilating and core, so that would require $s$ to be annihilating, in which case $\hat{M}_1=\hat{M}_2$ and we can refer to the previous case.
\end{itemize}

\noindent To avoid repetition we will from now on not mention the case when the sets of elements acted upon are disjoint, as it's clear that in that case the matrices commute.

\begin{itemize}
\item $i=1, j=3$: In this case we add two annihilating elements $g_1,g_2$ and delete two complementary elements $h_1,h_2$.
If $g_1=h_1$ the fact that $h_1,h_2$ are complementary reduces to $b(h_2,k)=b(s,k)$ for all $k$, so that $h_2$ is a core element.
This means that overall we changed a core element to an annihilating one, hence $\hat{M}_3^{-1}\circ \hat{M}_1\sim \hat{M}^{-1}_w$.
If the pairs are the same we have once again that $s$ is annihilating, and that our matrix is isomorphic to what it was before applying $\hat{M}_3^{-1}\circ\hat{M}_1$.

\item $i=2, j=3$: Similar to the previous case, except we end up with $\hat{M}_w$.

\item $i=3, j=1$: We add two complementary elements $\{g_1,g_2\}$ and then remove two annihilating elements $\{h_1,h_2\}$.
If $g_1=h_1$ this means that $g_2$ is a core element, so overall we added a core element and removed an annihilating one.
This means $\hat{M}_1^{-1}\circ \hat{M}_3\sim \hat{M}_w$.
If the pairs agree once again the matrix is unchanged up to isomorphism (and  $s$ is annihilating).

\item $i=3, j=2$: Similar to the previous case, except we end up with $\hat{M}_w^{-1}$.

\item $i=j=3$: this case is unchanged from the proof in \cite{5}.

\item $i=w, j=1$: $\hat{M}_w$ changes an element $g$ from annihilating to core, and $\hat{M}_1^{-1}$ removes two annihilating elements $\{h_1,h_2\}$.
Clearly $g$ cannot coincide with $h_i$, so the two operations commute.

\item $i=1, j=w$: once again the elements cannot coincide, so the operations trivially commute.

\item $i=w, j=2$: $\hat{M}_w$ changes an element $g$ from annihilating to core, and $\hat{M}_2^{-1}$ removes two core elements $\{h_1,h_2\}$.
If $g=h_1$ the overall result was removing the annihilating element $g$ and the core element $h_2$.
But this can be achieved by an application of $\hat{M}_3^{-1}$ (since the elements $g, h_2$ are complementary).

\item $i=2, j=w$: as above, if $g=h_1$ the net result of $\hat{M}_w^{-1}\circ \hat{M}_2$ was to add a core element ($h_2$) and an annihilating one ($g$), which we can do through an application of $\hat{M}_3$.

\item $i=w, j=3$: We change an element $g$ from annihilating to core, then remove two complementary elements $\{h_1, h_2\}$.
If $g=h_1$ it follows that $h_2$ must be annihilating, and since $g$ was annihilating to start with we get $\hat{M}_3^{-1}\circ \hat{M}_w\sim \hat{M}_1^{-1}$ (removing $g, h_2$).

\item $i=3, j=w$: This adds two complementary elements $h_1, h_2$, then changes a core element $g$ to annihilating.
If $g=h_1$ this means that $h_2$ is annihilating, so the net result is equivalent to $\hat{M}_1$ (adding $g, h_2$).

\item $i=j=w$: it's the identity if they coincide, they commute if disjoint.

\end{itemize}

\noindent Let us now verify the second statement:

\begin{itemize}
\item $j=1$: $\hat{M}_w\circ\hat{M}_1$ adds two annihilating elements $\{g_1, g_2\}$ and changes an element $h$ from annihilating to core.
If $h=g_1$ then the net result is adding a core element and an annihilating one, so $\hat{M}_w\circ \hat{M}_1\sim\hat{M}_3$.

\item $j=2$: the elements introduced by $\hat{M_2}$ cannot be acted upon by $\hat{M}_w$, so the two terms commute.

\item $j=3$: we add complementary $\{g_1, g_2\}$, then change $h$ from annihilating to core. 
If $h=g_1$ then $g_2$ must be a core element, so $\hat{M}_w\circ \hat{M}_3\sim \hat{M}_2$.

\end{itemize}

\noindent Note that a similar statement can be made for $\hat{M}_j^{-1}\circ \hat{M}_w^{-1}$, where the $\hat{M}_w^{-1}$ can be moved to the left.
This proves the claim and completes the proof of Lemma \ref{lemma1}.

\end{proof}

\begin{theorem}[see \cite{5}, Theorem 7.1.1]
If two framed virtual strings are homotopic, then their based matrices are homologous.
\end{theorem}

\begin{proof}
To prove the theorem we need to see how the three framed flat Reidemeister moves affect the based matrix.
The proof is analogous to the proof in \cite{5}, except that framed Reidemeister move one adds two double points instead of one; indeed, the analogy to Turaev's proof is what motivated the definitions of $\hat{M}_1,\hat{M}_2$.
\end{proof}

To generalize our framed virtual strings to the singular case, we keep the definitions given in \cite{2} of annihilating-like, core-like and the singularity switch, and we replace the Turaev elementary extensions with ours.

\begin{definition}
A singular based matrix is a quadruple $(G, s, d, b\colon G\times G\to H)$, where $H$ is an abelian group, $G$ is a finite set, $s, d\in G$ and the map $b$ is skew-symmetric.
\begin{itemize}
\item We call $g\in G\setminus \{s,d\}$ annihilating if $b(g,h)=0$ for all $h\in G$;
\item We call $g\in G\setminus \{s,d\}$ a core element if $b(g,h)=b(s,h)$ for all $h\in G$;
\item We call $g_1,g_2\in G\setminus \{s,d\}$ complementary if $b(g_1,h)+b(g_2,h)=b(s,h)$ for all $h\in G$.
\item We call $g\in\{s,d\}$ annihilating-like if $b(g,h)=0$ for all $h\in G$.
\item We say $d$ is core-like if $b(s,h)=b(d,h)$ for all $h\in G$.
\end{itemize}
\end{definition}

\begin{definition}The elementary operations on SBMs are as follows:
\begin{itemize}
\item $\hat{M}_1$ changes $(G,s,d,b)$ into $(G\coprod \{g_1,g_2\}, s, d, b_1)$ such that $b_1$ extends $b$ and $g_1, g_2$ are annihilating;
\item $\hat{M}_2$ changes $(G,s,d,b)$ into $(G\coprod \{g_1,g_2\}, s, d, b_2)$ such that $b_2$ extends $b$ and $g_1, g_2$ are core elements;
\item $\hat{M}_3$ changes $(G,s,d,b)$ into $(G\coprod \{g_1,g_2\}, s, d, b_3)$ such that $b_3$ is any skew-symmetric map extending $b$ in which $g_1, g_2$ are complementary.
\item $\hat{M}_w$ changes $(G,s,d,b)$ into $(G,s,d,b')$ by acting on an annihilating element $g\in G\setminus \{s,d\}$ and changing it into a core element.
\item Suppose that there is $g\in G$ such that $b(g,h)+b(d,h)=b(s,h)$ for all $h$. Then $N$ changes $(G,s,d,b)$ into $(G,s,g,b)$.
\end{itemize}
\end{definition}\

\noindent The notions of isomorphic, primitive and homologous singular based matrices are analogous to the unframed case, see Definition \ref{SBMiso}

\begin{lemma}
Every SBM is obtained from a primitive SBM by elementary extensions and singularity switches.
\end{lemma}

\begin{proof}
Reduce the cardinality of $G$ as much as possible by applying $\hat{M}_1^{-1}$, $\hat{M}_2^{-1}$, $\hat{M}_3^{-1}$.
If no more inverse elementary extensions can be applied, try applying $\hat{M}^{\pm 1}_w$ or $N$, then see if you can reduce the cardinality of $G$ further.
Because the cardinality of $G$ is monotonically decreasing under this process, it will at some point terminate.
The resulting SBM is primitive.
\end{proof}

Recall the moves $D_{ij}$ from \cite{2}: $D_{12}$ transformed the distinguished element $d$ from annihilating-like to core-like, while $D_{21}$ transformed $d$ from core-like to annihilating-like. 
We can do the same operations in our framed theory, even if the formulas are slightly different: $D_{12}=\hat{M}_3^{-1}\circ N \circ \hat{M}_2$ and $D_{21}=\hat{M}_2^{-1}\circ N\circ \hat{M}_3$, where $\hat{M}_3^{\pm1}$ adds/removes a pair \{core, annihilating\}.

\begin{theorem}[see \cite{2}, Theorem 12]
Given two homologous primitive SBMs, we can obtained one from the other via an isomorphism, at most one of $D_{12}, D_{21},N$ and at most one $\hat{M_w}$.\end{theorem}

\begin{remark}
The reason why the statement is a little more complicated than the respective statement in \cite{2} is that a framed primitive based matrix can already have either one or two elements in its homology class, and the introduction of the distinguished chord also gives one or two elements in the homology class.
Since the two possibilities are independent from each other, a homology class of SBMs can have one, two or four primitive elements in it.
\end{remark}

\begin{proof}
Let $P,P'$ be primitive homologous SBMs.
By definition, this means that they're related by a finite sequence of elementary extensions, their inverses and singularity switches.
Since $\hat{M}_j$ for $j\in\{1,2,3,w\}$ act on SBMs the same way that they acted on based matrices, we can replace $\hat{M}_j^{-1}\circ \hat{M}_i$ by $\hat{M}_l\circ \hat{M}_k^{-1}$ or $\hat{M}_k^{\pm1}$.
Moreover, isomorphisms commute with elementary extensions and singularity switches, so we can just stack them all on one side.
Then, in a similar fashion to the proof of Lemma \ref{lemma1}, we claim the following

\begin{claim} 
\label{claim2}
We can rewrite the sequence of moves between $P,P'$ as
\begin{equation}(\hat{M}_i, N )\circ (\text{seq of }\hat{M}_w^{\pm1}, D_{12}, D_{21}, N)\circ (\text{ isom. }) \circ (\hat{M}_i^{-1}, N )P=P'\end{equation}
where $i\in\{1,2,3\}$.
\end{claim}

Assuming the claim holds, primitivity of $P,P'$ means that there can be no $\hat{M}_i^{\pm1}, N$ at either end of the sequence.
It is then enough to show that a sequence of $\hat{M}_w^{\pm1}, D_{ij}, N$ moves can be reduced to having at most one $\hat{M}_w$ and one of $D_{ij}, N$.
We already know from Lemma \ref{lemma1} that even powers $\hat{M}_w^{\pm1}$ can be removed, as we can express them in terms of $\hat{M}_i$s and use primitivity.
Clearly $\hat{M}_w^{\pm1}$ commutes with $D_{ij}$, as one is applied to the distinguished element and the other is applied to something that isn't the distinguished element.
$\hat{M}_w^{\pm1}$ either commutes with $N$ (if they act on distinct elements) or the composition (in either order) is equivalent to one of the $D_{ij}$.

It is also easy to check that the composition of $N$ with a $D_{ij}$ (in either order) yields $\hat{M}^{\pm1}_{w}$.
Finally, $N^2=id$, $D_{ij}\circ D_{ji}= id$ and $D_{ij}^2$ is only possible if $s$ is annihilating, in which case $D_{ij}$ is just an isomorphism.
Because everything reduces or commutes, the square powers of $\hat{M}_w, N, D_{ij}$ reduce to isomorphisms, and the process is monotonically decreasing, at the end of the process we can have at most one $\hat{M}_w$ and one of $D_{ij},N$.

This completes the proof of the theorem.

\end{proof}

\begin{proof}[Proof of Claim \ref{claim2}]
The proof roughly goes the same way as in \cite{2}.
We need to show that can stack terms of the form $\hat{M}_w^{\pm1}, D_{ij}$ to the middle of the sequence, eventually with some $N$ terms.

We start by noticing that $D_{ij}$ commutes with $\hat{M}_i^{\pm1}$, $i\neq w$.
This is somewhat obvious, as the net result of $D_{ij}$ is to only act on the distinguished element.
Moreover, the proof of Lemma \ref{lemma1} already shows that we can push the $\hat{M}_i^{-1}$ terms to the right and stack the $\hat{M}_w$ terms in the middle.
Thus we only need to check that any move of the form $\hat{M}_j^{-1}\circ N \circ \hat{M}_i$ that isn't equivalent to $D_{ij}$ or an isomorphism can be rewritten in the desired form, i.e. as a sequence of equal or shorter length in which inverses extensions happen before extensions.
This is enough to bring the sequence in the desired form.

\begin{itemize}
\item $i=j=1$: $N$ exchanges the role of $d$ with a complementary element $k$, $\hat{M}_1$ introduces $\{g_1,g_2\}$ and $\hat{M}_1^{-1}$ removes $\{h_1, h_2\}$, all annihilating elements.
If $k\notin\{g_1, g_2\}$ or $d\notin \{h_1, h_2\}$ we can commute $N$ with the respective extension, at which point we have $\hat{M}_1^{-1}\circ \hat{M_1}$ on either side of $N$ and we can apply lemma \ref{lemma1} to get the desired form.
Now suppose that the action of $N$ intertwines the actions of $\hat{M}_1^{\pm1}$, say $k=g_1$ and $d=h_1$, so we add two annihilating elements, swap one of them with $d$ and remove what was $d$ with another annihilating element.
Since $k=g_1$ is annihilating, $d$ must've been core-like; then the only way we can remove $d$ via $\hat{M}_1^{-1}$ is if $s$ is annihilating.
Because of that, moves $\hat{M}_1, \hat{M}_2$ can be replaced by $\hat{M}_3$ with the complementary pair being \{annihilating, core\}.
But then $\hat{M}_1^{-1}\circ N \circ \hat{M}_1\sim \hat{M}_3^{-1}\circ N \circ \hat{M}_1=D_{12}$, whose case has already been covered.

\item $i=j=2$: Similarly to the previous case, $N$ intertwines the two extensions if and only if $s$ is annihilating, at which point the composition  $\hat{M}_2^{-1}\circ N \circ \hat{M}_2\sim D_{21}$.

\item $i=1, j=2$: we have $\hat{M}_2^{-1}\circ N \circ \hat{M}_1$.
Once again if $N$ acts on an element disjoint from either $\hat{M}_1$ or $\hat{M}_2^{-1}$ we can move the extensions next to each other and use lemma \ref{lemma1} to conclude.
If $N$ intertwines the two operations $d$ was core-like at the beginning, and the net result of the operations is equivalent to $D_{12}\circ\hat{M}_w^{-1}$.

\item $i=2, j=1$: similar to the previous case; if $N$ intertwines we have something equivalent to $\hat{M}_w\circ D_{21}$.

\item $i=1, j=3$: If $N$ intertwines we get $D_{12}$, otherwise they commute and we can apply lemma \ref{lemma1}.

\item $i=3, j=1$: Similar to the previous case, we get $D_{21}$ instead.

\item $i=2, j=3$: If $N$ intertwines, $\hat{M}_3^{-1}\circ N\circ \hat{M}_2^{-1}$ is equivalent to $D_{21}$.

\item $i=3, j=2$: Similar to the previous case, we get $D_{12}$ instead.

\item $i=j=3$: This case is unchanged from \cite{2}.

\end{itemize}

\end{proof}

\begin{cor}
There can be one, two or four primitive matrices in a homology class of SBMs.
\end{cor}

We can now look at the framed singular virtual string associated to a framed flat virtual knot with exactly one double point.
The construction is exactly the same as the unframed case: the arrow corresponding to the double point is the preferred element, we do not allow framed Reidemeister one moves that involve the double point, and we add the extra move (s-ii) as in \cite{2} (it will correspond to the singularity switch).
We get the SBM associated to the framed virtual string again using the intersection index and the formula $b(e,f)=ab\cdot cd+\epsilon$.
We pick the matrix representation of $b$ in which $s$ is the first row and $d$ is the last row.

\begin{theorem}
If two framed singular virtual strings are homotopic, their SBMs are homologous.
\end{theorem}

\begin{proof}
We need to see how the Reidemeister moves and (s-ii) act on the virtual string.
But since the analogous theorem holds in the non-singular case, and Reidemeister moves do not involve the singular crossings, they satisfy the required condition.
Moreover, the move (s-ii) acts precisely as the singularity switch $N$. Thus the result holds.\end{proof}

\subsection{Based matrices for long virtual strings}
\label{basedmatriceslongstrings}

To extend based matrices to flat long virtual knots and long virtual strings, it's useful for our purposes to reflect on the origin of the definitions in the closed case.
What Turaev was thinking of when he defined based matrices was the homological intersection matrix that's associated to a flat virtual knot $K$.
We know that we can find a minimal surface on which the knot lives; moreover, we can construct said surface using the band-pass presentation of the knot.
If $\Sigma$ is the minimal surface, we can look at the map $H_1(\Sigma)\times H_1(\Sigma)\to \Z$ determined by the orientation of $\Sigma$.
This bilinear pairing is skew-symmetric.

We also know that our flat virtual knot $K$ on $\Sigma$ is a deformation retract of $\Sigma$ itself.
Looking at $K$ as a CW-complex $\Gamma$, the inclusion $H_1(\Gamma)\to H_1(\Sigma)$ is an isomorphism, and since $\chi(\Gamma)=-m$ we have that $H_1(\Sigma)$ is a free abelian group of rank $m+1$.
We determine a preferred basis for $H_1(\Sigma)$ in the following way: we let $s=[K]\in H_1(\Sigma)$, and for every classical crossing of $K$ we do an oriented smoothing (which gives us a 2-component link), take the right-hand component and consider its homology class. 
We can then compute what the matrix looks like in this basis, getting the nicer, purely combinatorial formulas of \cite{5} in terms of the arrows on the virtual string.

Knowing how to associate a matrix to the virtual string, it is somewhat immediate to see that the elementary extension moves are deduced from the Reidemeister moves, see Fig. \ref{elementaryextensionsasreidemeister}.

\begin{figure}[!h]
\centering
\includegraphics[scale=.15]{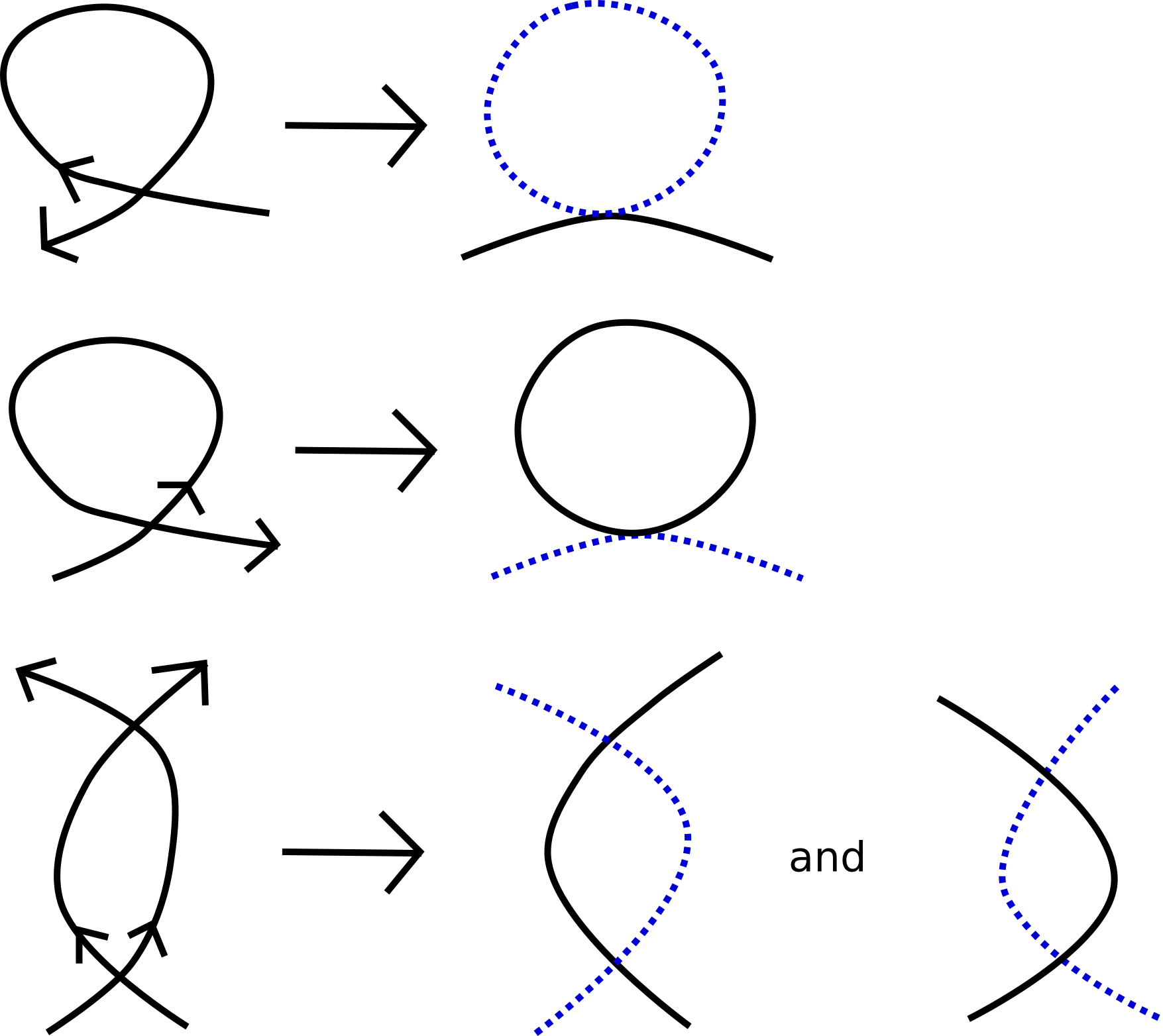}
\caption{The effect of Reidemeister moves on the basis of $H_1(\Sigma)$ and $H_1(\Sigma, C)$. The dotted cycles denote the elements that get added to the basis.}
\label{elementaryextensionsasreidemeister}
\end{figure}

Let us now consider the case of long virtual strings.
We want to create a correspondence between long virtual strings and based matrices, similarly to what we already have in the closed virtual string case.
We will mirror the above construction using the interpretation of long virtual knots as knots in thickened surfaces with boundary.
The one main difference from the closed knot case comes from the fact that the homology class of $K$ is no longer a cycle in $H_1(\Sigma)$.
To solve this issue, if $C$ is the collection of distinguished boundary components of $\Sigma$, we will consider the relative homology $H_1(\Sigma,C)$ instead of the absolute $H_1(\Sigma)$.

Using this trick the construction above falls smoothly into place: we get that $H_1(\Sigma, C)$ is isomorphic to $H_1(\Gamma, C)$, so that $H_1(\Sigma, C)$ is a free abelian group of rank $m+1$, where $m$ is the number of classical crossings of $K$.
We can still define the preferred basis in the same way.
It is worth pointing out however that there is another nice basis for the space, namely the one where, after the smoothing, we always pick the closed component.
While this choice has no impact on the theory of based matrix (we go from one basis to another by replacing $[e]$ for $s-[e]$, if $[e]$ is the homology class of an open component), it may be useful for other applications to pick the basis where all the cycles but one are closed curves.

Let us now discuss the validity of the purely combinatorial formulas for the entries of the based matrix.
Recall that a virtual string is the flat version of a Gauss diagram.
In the following let $ab$ denote the oriented arc of the core circle of the virtual string going from $a$ to $b$.
We denote by $\omega(ab)$ the arc of our long knot going from the preimage of $a$ to the preimage of $b$ (if $ab$ does not contain the distinguished point at infinity) or the union of the arcs going from infinity to $b$ and from $a$ to infinity (if $ab$ contains the distinguished point). 

Given an arrow in our long virtual string $e=(a,b)$, $a$ denotes the tail and $b$ denotes the head of the arrow.
Then $e$ gives the basis element $[e]=\omega(ab)\in H_1(\Sigma, C)$.
Its complementary element (the one we get if we take the left component after the smoothing) will be denoted $[e]^*=\omega(ba)$.
Note that $[e]^*=s-[e]$.
By definition of the homological intersection matrix, $b([e],[e]^*)=i(e)$, where $i(e)$ is the intersection index (defined in section \ref{polyinv}), computed on the tangle $[e]\cup [e]^*$ in which we set $[e]$ to be the first component, see Fig. \ref{intindexbasedmatrix}.

\begin{figure}[!h]
\centering
\includegraphics[scale=.12]{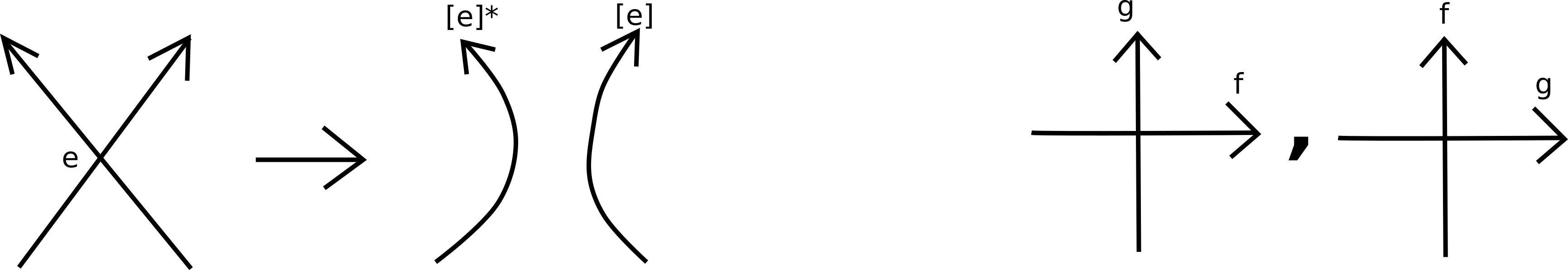}
\caption{On the left: How the homology classes $[e], [e]^*$ relate to the flat crossing $e$. On the right: two terms contributing respectively $+1$ and $-1$ to the computation of $b(f,g)$, for $f,g\in H_1(\Sigma, C)$. Putting both pictures together and recalling Definition \ref{def1} we see that $b([e], [e]^*)=i(e)$.}
\label{intindexbasedmatrix}
\end{figure}

We can compute this purely in terms of the virtual string by counting the number of arrows with tail in $ab$ and head in $ba$ minus the number of arrows with head in $ab$ and tail in $ba$, as (by definition of a virtual string) an arrow with tail in component $1$ and head in component $2$ looks like the positive crossing of Fig. \ref{flatlinksign}.
But then 
\begin{equation}b([e],[e]^*)=b([e], s-[e])=b([e],s)-b([e],[e])=b([e],s).\end{equation}
This establishes that $b([e],s)=i(e)$, as in the closed case.

For distinct points $a,b$ on the core circle of the virtual string, denote by $(ab)^\circ$ the arc $ab\setminus\{a,b\}$.
Denote by $ab\cdot cd\in\Z$ the number of arrows with tail in $(ab)^\circ$ and head in $(cd)^\circ$ minus the number of arrows with tail in $(cd)^\circ$ and head in $(ab)^\circ$.

\begin{lemma} For two arrows $e=(a,b)$, $f=(c,d)$ we have $$b([e], [f])=ab\cdot cd+\varepsilon,$$where $\varepsilon=0$ if the arrows are unlinked, $\varepsilon=1$ if $f$ links $e$ positively and $\varepsilon=-1$ if $f$ links $e$ negatively, see Fig. \ref{linkingarrows}.
\end{lemma}

The proof of this lemma is unchanged from \cite{5}, as the argument mainly rests on computing the intersection number through a local pushoff.

\subsection{Results for based matrices.}
The previous section illustrated how to associate a based matrix to a long virtual string, and that two homotopic long virtual strings yield homologous based matrices.
However, all the results of \cite{2}, \cite{5} and section \ref{framedmatricessection} only rely on the structure of the based matrices themselves, and not on their relation to virtual strings.
All the results regarding based matrices and their structure will thus hold in our case as well. 
We summarize all those results into the following definitions and theorem.

\begin{definition}[Based matrices and their homology]
A based matrix is a triple $(G,s,b\colon G\times G\to H)$ where $G$ is finite, $s\in G$ and $b$ is skew-symmetric.
Some elements get assigned a specific name, as follows:

\begin{itemize}
\item We call $g\in G\setminus \{s\}$ annihilating if $b(g,h)=0$ for all $h\in G$;
\item We call $g\in G\setminus \{s\}$ a core element if $b(g,h)=b(s,h)$ for all $h\in G$;
\item We call $g_1,g_2\in G\setminus \{s\}$ complementary if $b(g_1,h)+b(g_2,h)=b(s,h)$ for all $h\in G$.
\end{itemize}

We call elementary extensions of a based matrix the following three operations:

\begin{itemize}
\item $M_1$ changes $(G,s,b)$ into $(G\coprod \{g\}, s, b_1)$ such that $b_1$ extends $b$ and $g$ is annihilating;
\item $M_2$ changes $(G,s,b)$ into $(G\coprod \{g\}, s, b_2)$ such that $b_2$ extends $b$ and $g$ is a core element;
\item $M_3$ changes $(G,s,b)$ into $(G\coprod \{g_1,g_2\}, s, b_3)$ such that $b_3$ is any skew-symmetric map extending $b$ in which $g_1, g_2$ are complementary.
\end{itemize}

Two based matrices  $(G,s,b)$ and $(G', s',b')$ are isomorphic if there is a bijection $G\to G'$ such that $s\mapsto s'$ and $b$ becomes $b'$.
A based matrix is primitive if it cannot be obtained from another based matrix by elementary extensions.
Two based matrices are homologous if they can be obtained from each other by applying a finite number of elementary extensions and their inverses.
\end{definition}

\begin{definition}[Singular based matrices and their homology]\label{sbms}
A singular based matrix (SBM) is a quadruple $(G,s,d,b\colon G\times G\to H)$ where $G$ is finite, $s,d\in G$ and $b$ is skew-symmetric.
Some elements get assigned a specific name, as follows:

\begin{itemize}
\item We call $g\in G\setminus \{s,d\}$ annihilating if $b(g,h)=0$ for all $h\in G$;
\item We call $g\in G\setminus \{s,d\}$ a core element if $b(g,h)=b(s,h)$ for all $h\in G$;
\item We call $g_1,g_2\in G\setminus \{s,d\}$ complementary if $b(g_1,h)+b(g_2,h)=b(s,h)$ for all $h\in G$.
\item We call $g\in\{s,d\}$ annihilating-like if $b(g,h)=0$ for all $h\in G$.
\item We say $d$ is core-like if $b(s,h)=b(d,h)$ for all $h\in G$.
\end{itemize}

Homology of singular based matrices is obtained by considering the elementary extensions of the non-singular case and the following extra extension:

\begin{itemize}
\item Suppose that there is a $g\in G$ such that $b(g,h)+b(d,h)=b(s,h)$ for all $h\in G$.
Then the singularity switch $N$ changes $(G,s,d,b)$ into $(G,s,g,b)$.
\end{itemize}
Two SBMs  $(G,s,d,b)$ and $(G', s',d', b')$ are isomorphic if there is a bijection $G\to G'$ such that $s\mapsto s'$, $d\mapsto d'$ and $b$ becomes $b'$.
A SBM is primitive if it cannot be obtained from another SBM by elementary extensions, even after applying the singularity switch operation.
Two SBMs are homologous if they can be obtained from each other by applying a finite number of elementary extensions, their inverses and singularity switches.
\end{definition}

\begin{definition}[Based matrices for framed long virtual strings]
When working with framed long virtual strings, either singular or non-singular, we only need to replace the elementary extensions $M_1, M_2$ and their inverses by the following three moves and their inverses:

\begin{itemize}
\item $\hat{M}_1$ changes $(G,s,b)$ into $(G\coprod \{g_1,g_2\}, s, b_1)$ such that $b_1$ extends $b$ and $g_1, g_2$ are annihilating;
\item $\hat{M}_2$ changes $(G,s,b)$ into $(G\coprod \{g_1,g_2\}, s, b_2)$ such that $b_2$ extends $b$ and $g_1, g_2$ are core elements;
\item $\hat{M}_w$ changes $(G,s,b)$ into $(G,s,b')$ by acting on an annihilating element $g$ and changing it into a core element.
\end{itemize}
An explanation of the three moves is shown in Fig. \ref{framedelementaryextensions}.
The definitions of primitivity and isomorphism are natural generalizations of those in Definition \ref{BMiso}.
\end{definition}

\begin{figure}[!h]
\centering
\includegraphics[scale=.145]{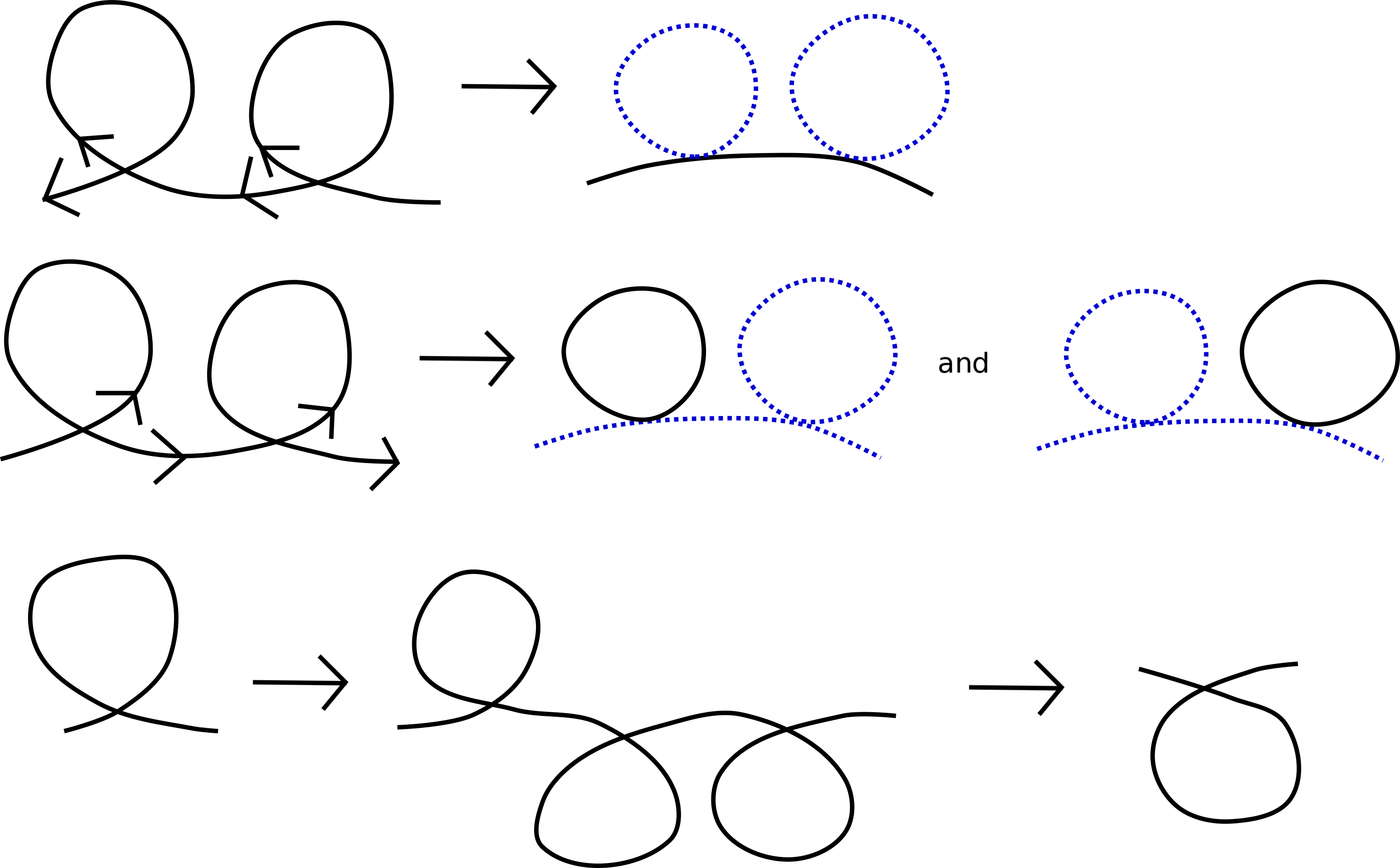}
\caption{The origin of the framed elementary extensions. The dotted cycles represent the elements added to the basis, and the third picture shows how $\hat{M}_w$ changes an annihilating element into a core one, or viceversa, by using $M_3^{-1}$ and $\hat{M}_1$ or $\hat{M_2}$.}
\label{framedelementaryextensions}
\end{figure}

\begin{theorem}[Based matrices all-in-one theorem] Any based matrix (framed/unframed, singular/non-singular) is obtained from a primitive based matrix by a finite number of elementary extensions.
Two homologous primitive based matrices of long virtual strings are isomorphic.
Two homologous primitive based matrices of long framed virtual strings are isomorphic or related by one $\hat{M}_w$ move.
Two homologous primitive based matrices of singular long virtual strings are isomorphic or related by one $N$ move.
Two homologous primitive based matrices of singular long framed virtual strings are isomorphic or related by one $\hat{M}_w$ and/or one $N$ move.

In any of the above situations and taking the appropriate elementary extensions, if two virtual strings are homotopic their based matrices are homologous.
\end{theorem}

All the proofs are unchanged from the closed virtual string case.

\subsection{Universality of the glueing invariant}
\label{universality}
Now that we defined the machinery we needed, we can go back to Theorem \ref{thm1} and present the pairs of long knots for which $S(K_1)=S(K_2)$ but $G(K_1)\neq G(K_2)$. 
\label{section5}
As we already mentioned, the glueing invariant is stronger than the smoothing invariant, since the glueing invariant is the universal Vassiliev invariant of order one.
To show it's strictly stronger we need to exhibit two virtual knots $K_1,K_2$ such that $S(K_1)=S(K_2)$ but $G(K_1)\neq G(K_2)$.

Let's start with the framed, closed case.
These are modified versions of the ones used in \cite{2}.
The reason why that pair doesn't work anymore is that it relied on Reidemeister move one to cancel out terms and show that $S(K_1)=S(K_2)=0$.
With the framed version of Reidemeister one that cancellation doesn't happen anymore, as every term has an extra kink in one of the components.
We can solve this issue by adding a kink to both knots in a specific spot, as shown in Fig. \ref{K1andK2}.

\begin{figure}[!h]
\centering
\includegraphics[scale=0.18]{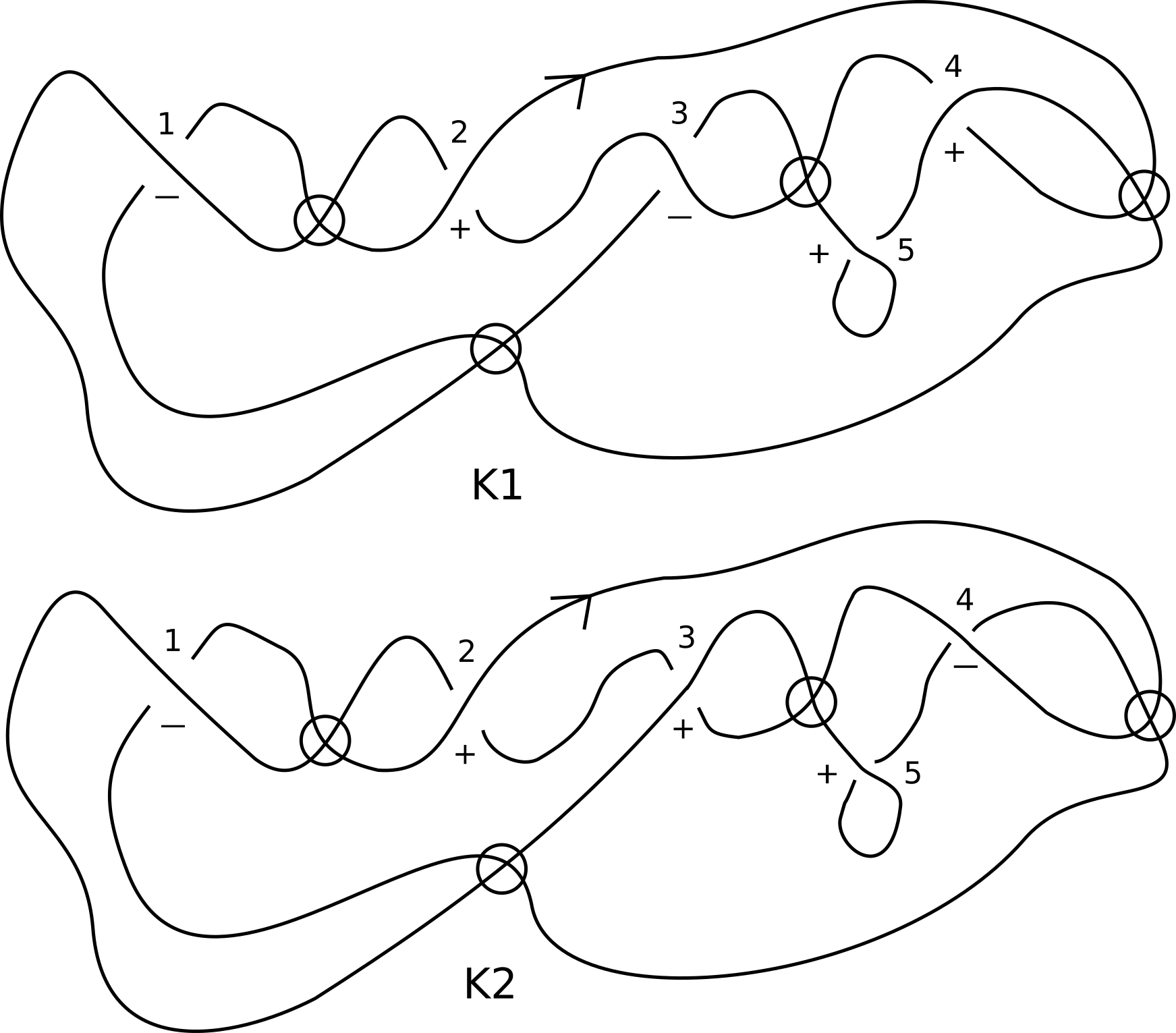}
\caption{The knots K1 and K2 in the closed, framed case.}
\label{K1andK2}
\end{figure}

\begin{claim}
$S(K_1)=S(K_2)$.
\end{claim}

\begin{proof}
$K_1$ and $K_2$ only differ at the crossings $3$ and $4$, so if we show that those two crossings contribute a same amount to $S(K_1), S(K_2)$ the claim will be proven.
Moreover, the two knots have the same writhe, so we only need to check that 
\begin{equation}-[(K_1)^3_{sm}]+[(K_1)^4_{sm}]=[(K_2)^3_{sm}]-[(K_2)^4_{sm}]\end{equation}
(according to our notation in section \ref{section1}).
But it's easy to see that $[(K_i)^3_{sm}]=[(K_i)^4_{sm}]$, as for crossing number $3$ we can first undo a virtual crossing and then use Reidemeister move two to remove two flat crossings, while for crossing $4$ we can (using flatness) move a kink around, then undo a virtual crossing and finally use Reidemeister move two to remove two flat crossings.
This completes the proof.
\end{proof}

Let us now construct the SBM associated with $K_1$ where we glued crossing $3$.
Again, it is similar to the one in \cite{2}, except there is a small isolated arrow at the bottom of the picture, see Fig. \ref{figure4}.

\begin{figure}[h]
\centering
\includegraphics[scale=.25]{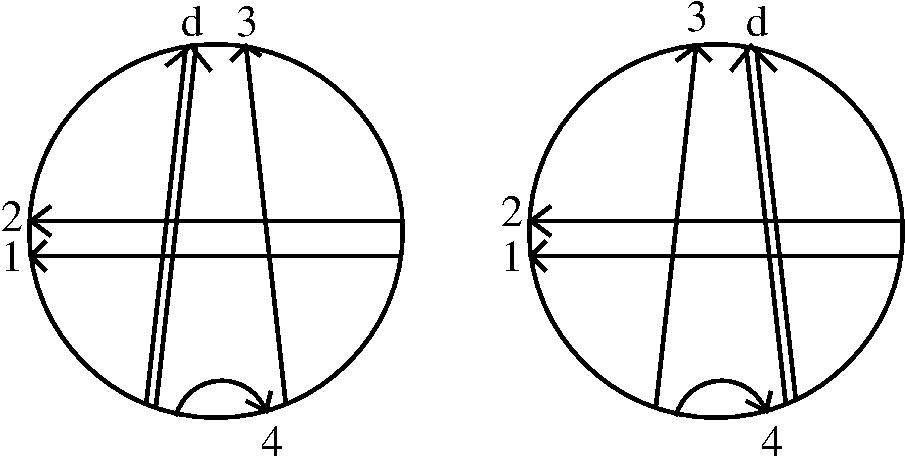}
\caption{The singular virtual strings of $K1$ where we respectively glued crossings $3$ and $4$ of Fig. \ref{K1andK2}.}
\label{figure4}
\end{figure}

Since the construction has already been mentioned in the paper, we leave to the reader the exercise of checking that the correct SBM for $K1$ where crossing $3$ was glued is the following:

\begin{equation}
\begin{pmatrix}
0&2&2&-2&0&-2\\
-2&0&0&-2&0&-3\\
-2&0&0&-1&0&-2\\
2&2&1&0&0&0\\
0&0&0&0&0&0\\
2&3&2&0&0&0
\end{pmatrix}
\end{equation}

Note that this is just the matrix of \cite{2} with an extra annihilating element added to it, which is what we expected (since we simply added a kink to the original knot).
Similarly, the matrix of $K_1$ where crossing $4$ got glued is the following

\begin{equation}
\begin{pmatrix}
0&2&2&-2&0&-2\\
-2&0&0&-3&0&-2\\
-2&0&0&-2&0&-1\\
2&3&2&0&0&0\\
0&0&0&0&0&0\\
2&2&1&0&0&0
\end{pmatrix}
\end{equation}

Both matrices are primitive: neither $\hat{M}_1$ nor $\hat{M}_2$ can be applied, and no two rows add up to $s$, so $\hat{M}_3$ and $N$ cannot be applied either.
We could at best apply $\hat{M}_w$ to crossing $4$, but we would still be unable to apply $\hat{M}_2$ because it cannot involve the $s$ row.
$d$ is clearly neither annihilating-like nor core-like in either matrix, and neither isomorphisms nor the $N$ move can relate the two SBMs.

\begin{remark}
Note that even if the two matrices only differ by the exchanged roles of the third and last row, we cannot use the $N$ move to exchange them because they don't add up to the values of the $s$ row.
\end{remark}

Since the SBMs are not homologous, the flat knots are not homotopic.
So the terms in $G(K_1)$ corresponding to crossings $3$ and $4$ do not cancel like they did in $S(K_1)$.
The corresponding terms in $G(K_2)$ have opposite signs, so in $G(K_1)-G(K_2)$ those terms add up, $[(K_i)^3_{glue}]$ with coefficient $-2$ and $[(K_i)^4_{glue}]$ with coefficient $+2$. 
Since these terms do not cancel, $G(K_1)-G(K_2)\neq 0$, hence $G(K_1)\neq G(K_2)$.
This ends the proof for the closed, framed case.

Before looking at the long virtual knot case, we should mention the following useful fact.

\begin{prop}
The (singular) based matrix associated to a long virtual knot only depends on the closure of the knot.
Two long virtual knots with the same closure yield the same (singular) based matrix.
\end{prop}

\begin{proof}
This result should not surprise the attentive reader, since the based matrix construction was purely combinatorial, based off of the virtual string and never involving the distinguished point at infinity.
The first row of the matrix is unchanged, as we compute it as the intersection index of the ordered tangle where, after the smoothing, we always take the right-hand component as component one (see Fig. \ref{smoothinggauss})
Every other entry of the matrix is obtained by counting arrows in certain arcs of the virtual string, and the presence of the point at infinity does not influence the counting in any way.\end{proof}

\begin{figure}[h]
\centering
\includegraphics[scale=.15]{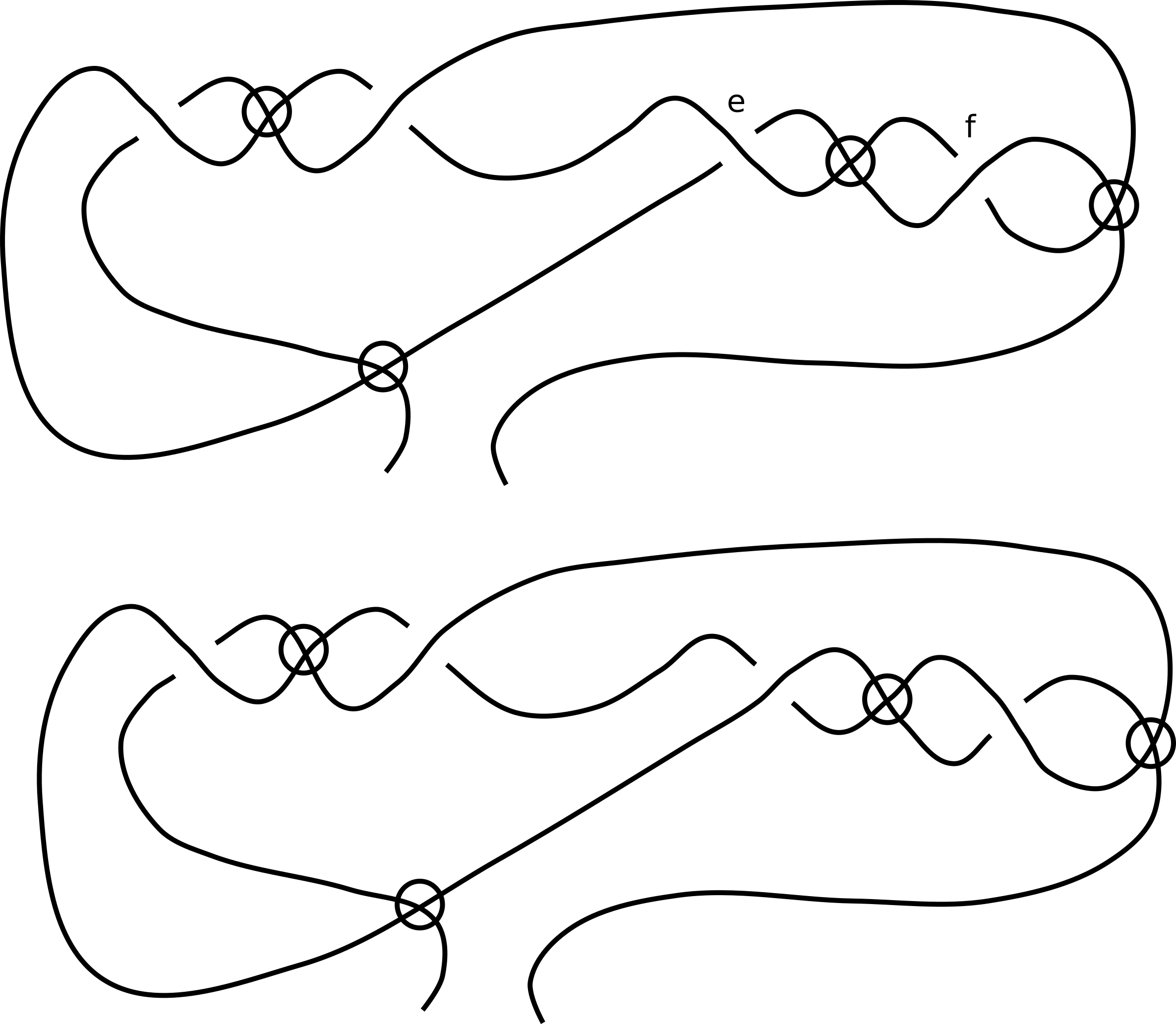}
\caption{The two knots distinguished by $G$ but not by $S$ in the long, unframed case.}
\label{unframedk1k2}
\end{figure}

Let's start with the unframed case; the two knots $K_1, K_2$ are pictured in figure \ref{unframedk1k2}.
These long knots are obtained from the ones in \cite{2}, Fig. 18 by picking a point at infinity.
These two knots only differ by the two crossings marked $e$ and $f$ in $K_1$, which get switched in $K_2$.
These crossings have opposite sign, and the two virtual tangles we get by smoothing either $e$ or $f$ are homotopic (we only need to do some untwisting using classical and virtual Reidemeister move one).
Thus the two knots have the same smoothing invariant value.

When computing the glueing invariant, the expression $G(K_1)-G(K_2)$ reduces to $2([(K_1)^f_{glue}]-[(K_1)^e_{glue}])$.
If we can show that these two flat classes are distinct (using singular based matrices) it will follow that $G(K_1)-G(K_2)\neq 0$.
But the respective based matrices are easily computed as follows:

\begin{equation}
\begin{pmatrix}
0&2&2&-2&-2\\
-2&0&0&-2&-3\\
-2&0&0&-1&-2\\
2&2&1&0&0\\
2&3&2&0&0
\end{pmatrix}
\hskip 1 in
\begin{pmatrix}
0&2&2&-2&-2\\
-2&0&0&-3&-2\\
-2&0&0&-2&-1\\
2&3&2&0&0\\
2&2&1&0&0
\end{pmatrix}
\end{equation}

These two singular based matrices are primitive, and not homologous.
It is an easy check that no inverse elementary extensions or singularity switches can be performed on either matrix.
In fact, the two matrices only differ by a swapping of the last two rows.
However, the last row represents the distinguished element $d$ of the singular based matrix, and the only operation we can perform on it is a singularity switch.
Since the last two rows don't add up to the top row, this is not an acceptable homology move, so the two matrices are non-homologous, and the two flat knots are non-homotopic.

In the framed case the argument is exactly the same, but we need to use the long knots pictured in Fig. \ref{framedk1k2}.

\begin{figure}[th]
\centering
\includegraphics[scale=.13]{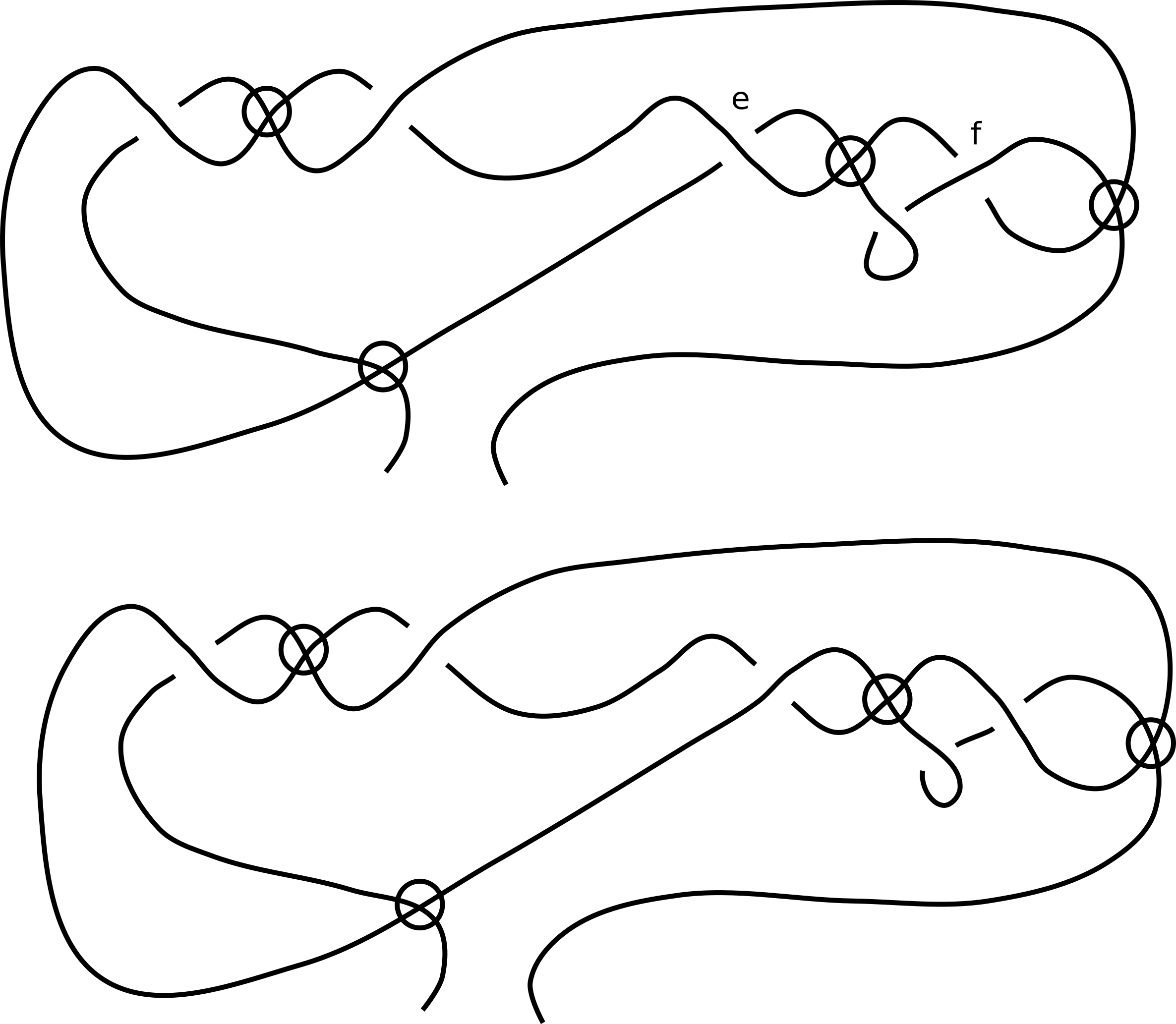}
\caption{The two framed knots distinguished by $G$ but not by $S$ in the long, framed case.}
\label{framedk1k2}
\end{figure}

Again, the two knots have the same smoothing invariant, as they only differ at the crossings $e$ and $f$ and smoothing at those crossings gives homotopic knots.
When looking at the glueing invariant $G(K_1)-G(K_2)-2([(K_1)^f_{glue}]-[(K_1)^e_{glue}])$, and we can distinguish the homotopy classes of those two knots using their singular based matrices.
A simple computations shows the based matrices to be the following:

\begin{equation*}
\begin{pmatrix}
0&2&2&-2&0&-2\\
-2&0&0&-2&0&-3\\
-2&0&0&-1&0&-2\\
2&2&1&0&0&0\\
0&0&0&0&0&0\\
2&3&2&0&0&0
\end{pmatrix}
\hskip 1 in
\begin{pmatrix}
0&2&2&-2&0&-2\\
-2&0&0&-3&0&-2\\
-2&0&0&-2&0&-1\\
2&3&2&0&0&0\\
0&0&0&0&0&0\\
2&2&1&0&0&0
\end{pmatrix}
\end{equation*}

Recall that we're in the framed case, so the moves $\hat{M}_1$ and $\hat{M}_2$ only remove pairs of annihilating or core elements.
This shows once again that these two matrices are primitive, and using the same techniques as the unframed case we can easily see that the two matrices are not homologous.
Thus $G(K_1)\neq G(K_2)$ in the framed case as well.

\section*{Acknowledgments}
I would like to thank the mathematics department at Dartmouth College, and my advisor Vladimir Chernov in particular, for the physical and moral support during the time it took to come up with these results.
I am also thankful to Dror Bar-Natan and Masahico Saito, whose comments after a presentation were the seed that let to some of the results in this paper, namely the results for long knots and the ordered polynomial invariant.

\end{document}